\documentclass[11pt]{amsart}
\usepackage{hyperref}
\usepackage{fullpage}
\usepackage{amsmath,amsthm,amssymb}
\usepackage{dsfont}
\usepackage{diagbox}
\usepackage{enumitem}
\usepackage{caption, subcaption}
\usepackage{tikz,color}
\usetikzlibrary{patterns}

 \newtheorem{theorem}{Theorem}[section]
 \newtheorem{corollary}[theorem]{Corollary}

 \newtheorem{lemma}[theorem]{Lemma}

 \newtheorem{proposition}[theorem]{Proposition}
 
 \theoremstyle{definition}
 \newtheorem{example}[theorem]{Example}
 
 \newtheorem{problem}[theorem]{Problem}
 \newtheorem{definition}[theorem]{Definition}
 
 \newtheorem{remark}[theorem]{Remark}

\tikzset{
  hatch size/.store in=\hatchsize,
}
 
\definecolor{mygreen}{rgb}{0.1,0.8,0.1}
\definecolor{mygray}{rgb}{0.7,0.7,0.7}

\newcommand{\beq}{\begin{equation}}
\newcommand{\eeq}{\end{equation}}

\newcommand{\beqn}{\begin{eqnarray}}
\newcommand{\eeqn}{\end{eqnarray}}

\newcommand{\tdot}[3]{\draw [fill=black, color=#3] (#1,#2) circle [radius=0.25];}

\newcommand{\OO}{\mathcal{O}}
\newcommand{\In[1]}{\mathrm{in}^{\OO}_{#1}}
\newcommand{\Out[1]}{\mathrm{out}^{\OO}_{#1}}
\newcommand{\orighta}{\xrightarrow{\OO}}
\newcommand{\olefta}{\xleftarrow{\OO}}

\newcommand{\CNAT[1]}{\mathrm{CNAT}\left( #1 \right)}
\newcommand{\cnat[1]}{\mathrm{cnat}\left( #1 \right)}

\newcommand{\Insert[1]}{\mathrm{Insert}_{#1}}
\newcommand{\Prune[1]}{\mathrm{Core}_2\left( #1 \right)}

\newcommand{\des[1]}{\mathrm{Desc}\left( #1 \right)}
\newcommand{\ER[1]}{\mathrm{EmptyRow}\left( #1 \right)}

\newcommand{\dec[1]}{\mathrm{dec}_{#1}}

\newcommand{\dgr[1]}{\mathrm{deg}_{#1}}

\newcommand{\concat}{\! \cdot \!}

\newcommand{\T}{\mathcal{T}}

\newcommand{\N}{\mathbb{N}}

\begin{document}

\title{Complete non-ambiguous trees and associated permutations: new enumerative results}

\author{Thomas Selig and Haoyue Zhu}
\address{Department of Computing, School of Advanced Technology, Xi'an Jiaotong-Livepool University, 111, Ren'ai Road, Suzhou 215123, China} 
\email[T.~Selig (corresponding author)]{Thomas.Selig@xjtlu.edu.cn}
\email[H.~ Zhu]{Haoyue.Zhu18@student.xjtlu.edu.cn}

\date{\today}

\begin{abstract}
We study a link between complete non-ambiguous trees (CNATs) and permutations exhibited by Daniel Chen and Sebastian Ohlig in recent work. In this, they associate a certain permutation to the leaves of a CNAT, and show that the number of $n$-permutations that are associated with exactly one CNAT is $2^{n-2}$. We connect this to work by the first author and co-authors linking complete non-ambiguous trees and the acyclic orientation number of the associated permutation graph. This allows us to prove a number of conjectures by Chen and Ohlig on the number of $n$-permutations that are associated with exactly $k$ CNATs for various $k > 1$, via various bijective correspondences between such permutations. We also exhibit a new bijection between $(n-1)$-permutations and CNATs whose permutation is the decreasing permutation $n(n-1)\cdots1$. This bijection maps the left-to-right minima of the permutation to dots on the top row of the corresponding CNAT, and descents of the permutation to empty rows of the CNAT.
\\[2mm]
 \end{abstract}
 
\maketitle




\section{Introduction}\label{sec:intro}

Non-ambiguous trees (NATs) were originally introduced by Aval \textit{et al.}\ in~\cite{ABBS} as a special case of the tree-like tableaux from~\cite{ABN}. The combinatorial study of these objects was further developed in~\cite{ABDHL}, which included a generalisation of NATs to higher dimensions. In~\cite{DSSS2}, the authors described a multi-rooted generalisation of complete NATs (CNATs), and linked this to the so-called Abelian sandpile model.

We can associate a permutation to a CNAT by keeping only its leaf dots (see Section~\ref{subsec:cnats_prelims}). This link between CNATs and permutations was first noted, to the best of our knowledge, in~\cite{DSSS2} (see Section~4 in that paper for more details). Chen and Ohlig~\cite{CO} initiated the first in-depth combinatorial study of this relationship. In particular, they characterised the set of permutations $\pi$ which are associated to a unique CNAT. More generally, they looked into the number of $n$-permutations with $k$ associated CNATs, for $k \geq 1$, and provided a number of conjectures on the enumeration of such permutations. They also studied in detail so-called \emph{upper-diagonal} CNATs, which are CNATs where the associated permutation is the decreasing permutation $n(n-1)\cdots1$, using these to prove a conjecture of Laborde-Zubieta on occupied corners of tree-like tableaux~\cite{Zub}. Finally, they introduced a natural statistic, \emph{determinant}, on CNATs, and conjectured that CNATs with determinant $+1$ and those with determinant $-1$ are equinumerous. This conjecture was solved very recently by Aval~\cite{AvalDet} through a bijective approach.

In this paper, we will further develop the study initiated in~\cite{CO}, using results from~\cite{DSSS2}. In that work it was shown that the number of CNATs associated with a given permutation $\pi$ is equal to the number of \emph{minimal recurrent configurations} for the Abelian sandpile model on the \emph{permutation graph} $G_{\pi}$ of $\pi$. These are counted by certain acyclic orientations of the graph (see Theorem~\ref{thm:cnat_tutte} in this paper). This enumerative result allows us to solve a number of conjectures in~\cite{CO}, via various bijective correspondences between permutations with a given number of CNATs.

We also add to the study of upper-diagonal CNATs by providing a new bijection between upper-diagonal CNATs of size $n$ and permutations of length $n-1$ (see Theorem~\ref{thm:bij_cnat_perm}). This new bijection has the following added benefits. Firstly, it is direct, whereas the bijection in~\cite{CO} used intermediate objects called tiered trees (see~\cite{DGGS}). Secondly, it maps certain statistics of the upper-diagonal CNATs such as the number of top-row dots, or the number of empty rows, to well-known statistics of the corresponding permutation.

Our paper is organised as follows. In Section~\ref{sec:prelim}, we recall some necessary definitions of, and notation on, the combinatorial objects that will be considered in this paper. These include graphs, acyclic orientations, permutations, and CNATs with associated permutations. We also show a number of useful preliminary results that will be useful in the remainder of the paper. At the end of the section we provide the key connection between CNATs and the acyclic orientation number of an associated graph (Theorem~\ref{thm:cnat_tutte}). In Section~\ref{sec:upper-diag}, we focus on one specific family of CNATs, called \emph{upper-diagonal CNATs}. We introduce a concept of \emph{labelled CNAT}, and describe a bijection between labelled CNATs and permutations of the label set (Theorem~\ref{thm:bij_cnat_perm}), which preserves certain statistics of both objects. This bijection then specialises to a bijection between upper-diagonal CNATs of size $n$ and permutations of length $n-1$. The bijection relies on two operations on labelled CNATs, called \emph{top-row decomposition} and \emph{top-row deletion}. 

In Section~\ref{sec:cnat_count}, we focus on counting permutations according to their number of associated CNATs. Specifically, for $k, n \geq 1$ we are interested in the set $B(n,k)$ of permutations of length $n$ which are associated with exactly $k$ CNATs. The study of these sets was initiated in~\cite{CO}, and we continue that work here by proving a number of conjectures left by the authors. We begin in Section~\ref{subsec:B(n,1)_B(n+1,2)} by giving characterisations of the sets $B(n,1)$ and $B(n,2)$ in terms of the so-called \emph{quadrant condition} (Definition~\ref{def:quad_cond}, Propositions~\ref{pro:B(n,1)_quadrants} and \ref{pro:B(n,2)_characterisation}). This allows us to describe a simple bijection between the product set $\{2, \dots, n\} \times B(n,1)$ and $B(n+1,2)$ (Theorem~\ref{thm:Insert_B(n,1)toB(n+1,2)}), and deduce the enumerative formula for the latter (Corollary~\ref{cor:b(n,2)}). In Section~\ref{subsec:B(n,2)_B(n+1,3)} we use permutation \emph{patterns} to establish a bijection between the sets $B(n,2)$ and $B(n+1,3)$ (Theorem~\ref{thm:biject_B(n,2)toB(n+1,3)}). We then show in Section~\ref{subsec:B(n,5)} that the set $B(n,5)$ is always empty (Theorem~\ref{thm:b(n,5)=0}). Finally, in Section~\ref{subsec:max_b(n,k)} we consider the maximal value of $k$ such that $B(n,k)$ is non-empty. We show (Theorem~\ref{thm:max_b(n,k)}) that this is achieved for $k = (n-1)!$, and that the unique permutation in $B(n, (n-1)!)$ is precisely the decreasing permutation $n(n-1)\cdots 1$ whose CNATs were studied in detail in Section~\ref{sec:upper-diag}. We conclude our paper in Section~\ref{sec:conc} with a brief summary of our results as well as some open problems and directions for future research.


\section{Preliminaries}\label{sec:prelim}

In this section we introduce the various objects that will be studied and used throughout the paper, alongside the necessary notation. We also state and prove a number of useful preliminary results.

\subsection{Graphs and orientations}\label{subsec:graphs_prelims}

Throughout this paper all graphs considered are finite, undirected, and connected. For a graph $G$, we denote $V(G)$ and $E(G)$ its set of vertices and edges respectively, and write $G = (V(G), E(G))$ (here $E(G)$ is a multiset since $G$ may have multiple edges). An edge $e \in E(G)$ is a \emph{bridge} if removing $e$ disconnects the graph $G$. A \emph{loop} is simply an edge whose two end-points are the same. A graph $G$ is called \emph{simple} if it contains no loops or multiple edges. Two graphs $G = (V,E)$ and $G' = (V', E')$ are said to be \emph{isomorphic} if there exists a bijection $\phi : V \rightarrow V'$ such that for any vertices $v, w \in V$, we have $(v,w) \in E$ if, and only if, $(\phi(v), \phi(w)) \in E'$. Given a subset $V' \subseteq V(G)$, the \emph{induced} subgraph on $V'$ is the graph with vertex set $V'$ and edge set consisting of all edges in $E$ whose end-points are both in $V'$. We denote it $G\left[V' \right]$. More generally, a \emph{subgraph} of $G$ is a graph $G'$ such that $V(G') \subseteq V(G)$ and $E(G') \subseteq E(G)$. It is a \emph{spanning} subgraph if $V(G') = V(G)$.

A \emph{cycle} in the graph $G$ is a sequence of vertices $v_0, v_1, \ldots, v_n = v_0$ (for some $n \geq 3$) such that for every $i > 0$, $(v_{i-1}, v_i)$ is an edge of $G$, and the vertices $v_0, v_1, \ldots, v_{n-1}$ are all distinct. The integer $n$ is the \emph{length} of the cycle, and we will say that $G$ \emph{contains} a cycle of length $n$ (or $n$-cycle for short) if there exists such a cycle in $G$. The $n$-cycle (graph) $C_n$ is the graph consisting of a single cycle of length $n$, with no other vertices or edges. If $G$ contains a cycle $v_0, v_1, \ldots, v_n = v_0$ such that $G\left[ \{v_0, \ldots, v_{n-1}\} \right]$ is the $n$-cycle $C_n$, we say that $G$ \emph{induces} an $n$-cycle. A \emph{tree} is a (connected) graph containing no cycle of any length. A \emph{spanning tree} of a graph $G$ is a spanning subgraph which is a tree. The \emph{$n$-path} (graph) is the graph with vertex set $V(G) = \{ v_1, \ldots, v_n \}$ and edge set $E(G) = \{ (v_i, v_{i+1}) \}_{1 \leq i \leq n-1}$ (equivalently the $n$-cycle with any single edge removed).

An \emph{orientation} $\OO$ of a graph $G$ is the assignment of a direction to each edge of $G$. Given an orientation $\OO$ of $G$, and an edge $(v,w)$, we write $v \orighta w$ to indicate that the edge is directed from $v$ to $w$ in the orientation $\OO$, and $v \olefta w$ when it is directed from $w$ to $v$. We also write $\In[v]$, resp.\ $\Out[v]$ for the number of incoming edges (edges $v \olefta w$), resp.\ outgoing edges (edges $v \orighta w$), at $v$ in the orientation $\OO$. A vertex $v$ is a \emph{sink}, resp.\ \emph{source}, of an orientation $\OO$ if $\In[v] = \dgr[v]$ (all edges are incoming), resp.\ $\Out[v] = \dgr[v]$ (all edges are outgoing).

An orientation is \emph{acyclic} if it contains no directed cycle, i.e.\ there is no sequence $v_0, \ldots, v_{n-1}$ of vertices such that $v_0 \orighta v_1 \orighta \cdots \orighta v_{n-1} \orighta v_0$. It is straightforward to check that an acyclic orientation must have at least one source and at least one sink. For $s \in V(G)$, we see that an acyclic orientation $\OO$ is \emph{$s$-rooted} if $s$ is the unique sink of $\OO$. 

\subsection{Acyclic orientation numbers}\label{subsec:tutte_prelims}

A key result that we will use in this paper links CNATs (see Section~\ref{subsec:cnats_prelims}) to a certain quantity related to acyclic orientations of graphs.

\begin{proposition}\label{pro:tutte_acyc_or}
Let $G$ be a graph, and $s \in V(G)$ a fixed sink vertex of $G$. Then the number of $s$-rooted acyclic orientations of $G$ does not depend on the choice of sink vertex $s$. We denote this number $a_G$, and call it the \emph{acyclic orientation number}.
\end{proposition}

\begin{proof}
Let $T_G$ denote the \emph{Tutte polynomial} of a graph $G$ (see e.g.~\cite[Chapter~2]{TutteBook} for definitions and more information on this important polynomial). Then we have that $T_G(1,0)$ is equal to the number of $s$-rooted acyclic orientations of $G$ for any fixed $s \in V(G)$ (see e.g.~\cite{Ber} for a bijective proof of this fact). Thus $a_G$ does not depend on the choice of sink vertex $s$.
\end{proof}

We first recall the famous deletion-contraction relation for the acyclic orientation number of a graph $G$. This is a simple specialisation of the more general result for the Tutte polynomial (see e.g.~\cite[Definition~2.7]{TutteBook}). Given a graph $G$ and an edge $e = (v, w) \in E(G)$, we denote by $G \setminus e$ the graph $G$ with edge $e$ \emph{deleted}. That is, $G \setminus e := \left( V(G), E(G) \setminus \{e\} \right)$. Similarly, $G \cdot e$ denotes the graph $G$ with edge $e$ \emph{contracted}. This is the graph $G$ in which the two end-points $v$ and $w$ are replaced by a single vertex $vw$, and edges incident to either $v$ or $w$ are made incident to $vw$ instead (with multiplicity preserved).

\begin{proposition}\label{pro:delete_contract}
Given a graph $G$ and an edge $e \in E(G)$, the acyclic orientation number $a_G$ satisfies the following recurrence:
\beq\label{eq:Tutte_poly}
a_G = \left\{
  \begin{array}{ll}
  a_{G \setminus e} + a_{G \cdot e} & \text{if } e \text{ is neither a bridge nor a loop},\\
  a_{G \cdot e} & \text{if } e \text{ is a bridge},\\
  0  & \text{if } e \text{ is a loop},\\
  1 & \text{if } G \text{ consists of a single vertex (and no edges)}.\\
  \end{array}
\right.
\eeq
\end{proposition}

We now state a number of useful structural results on how the acyclic number of a graph behaves with respect to certain graph constructions. We begin by characterising graphs whose acyclic orientation number is $0$ or $1$.

\begin{lemma}\label{lem:tutte_loop}
Let $G$ be a graph. Then we have $a_G = 0$ if, and only if, $G$ contains a loop. In particular, if $G$ is simple, we have $a_G \geq 1$.
\end{lemma}

\begin{proof}
If $G$ contains a loop, then Equation~\eqref{eq:Tutte_poly} implies that $a_G = 0$. Conversely, if $G$ is loop-free, fix some arbitrary total order $v_1, \ldots, v_n$ of the vertex set $V(G)$. Define an orientation $\OO$ of $G$ by $v_i \orighta v_j$ if, and only if, $i < j$. Since $G$ is loop-free, this does indeed assign a direction to each edge of $G$, and by construction $\OO$ is a $v_n$-rooted acyclic orientation. Proposition~\ref{pro:tutte_acyc_or} then implies that $a_G \geq 1$, as desired. 
\end{proof}

\begin{lemma}\label{lem:tutte_tree}
Let $G$ be a graph. Then we have $a_G = 1$ if, and only if, all cycles in $G$ have length $2$. In particular, if $G$ is simple, then $a_G = 1$ if, and only if, $G$ is a tree.
\end{lemma}

\begin{proof}
Suppose that $G$ is such that all its cycles have length $2$, and let $e = (v,w)$ be a multi-edge corresponding to such a cycle (i.e.\ $e$ has multiplicity at least $2$ in $G$). By definition, we have $a_G = a_{G \setminus e} +  a_{G \cdot e}$. But contracting $e$ in $G$ yields a loop at the new vertex $vw$, since $e$ was a multi-edge in $G$. Therefore $a_{G \cdot e} = 0$ by Lemma~\ref{lem:tutte_loop}, which implies that $a_G = a_{G \setminus e}$. Repeating this process until all mutli-edges have been removed, we are left with a tree $G'$ and $a_G = a_{G'}$. But now all edges in $G'$ are bridges, and we may contract them one-by-one until we reach a graph consisting of a single vertex. Applying Equation~\eqref{eq:Tutte_poly} throughout this process immediately yields $a_{G'} = 1$, as desired.

Conversely, suppose that $G$ contains a cycle of length other than $2$. As above, we may remove all multi-edges of $G$ without changing the acyclic orientation number. Moreover, if $G$ contains a loop, we have $a_G = 0$ by Lemma~\ref{lem:tutte_loop}.  We therefore assume, without loss of generality, that $G$ is simple and contains a cycle of length $k \geq 3$. Let $e$ be an edge in such a cycle. Note that $e$ is not a bridge. Moreover, since $G$ is simple, contracting $e$ yields a graph which is loop-free, while deleting $e$ yields a simple (connected) graph. Therefore, applying the Deletion-contraction Equation~\eqref{eq:Tutte_poly} at edge $e$, combined with Lemma~\ref{lem:tutte_loop}, we get $ a_G = a_{G \setminus e} + a_{G \cdot e} \geq 1 + 1 = 2$, as desired.
\end{proof}

Our next result can be thought of as a generalisation of Lemma~\ref{lem:tutte_tree} in the simple case. It shows that ``pruning'' a graph $G$ of its tree branches does not affect the acyclic orientation number. Given a graph $G$, we define the \emph{2-core} of $G$, denoted $\Prune[G]$, through the following algorithmic procedure.
\begin{enumerate}
\item\label{prune:init} Initialise $H = G$.
\item\label{prune:remove} If $H$ contains no vertex of degree $1$, move to Step~3. Otherwise, choose a vertex $v$ with degree $1$, set $H = H \setminus \{v\}$, and repeat Step~2.
\item\label{prune:output} Output $H := \Prune[G]$.
\end{enumerate}
Another way of viewing this procedure is that it removes ``tree branches'' that were attached at some vertices in the graph $G$. An example is shown in Figure~\ref{fig:pruning} below. Note that the output $\Prune[G]$ only depends on the choice of vertex $v$ in Step~\eqref{prune:remove} if at some stage of the process we reach a graph $H$ which is reduced to a single edge, which occurs in fact exactly when the original graph $G$ is a tree. In that case, removing either of the two vertices will yield a graph consisting of a single vertex, and we simply set $\Prune[G]$ to be any arbitrary vertex of the original graph.

\begin{figure}[ht]
\centering
\begin{tikzpicture}[scale=0.9]
\node [draw, circle] (1) at (0,0) {};
\node [draw, circle] (2) at (-1,1.5) {};
\node [draw, circle] (3) at (-3,1.5) {};
\node [draw, circle] (4) at (-4,0) {};
\node [draw, circle] (5) at (-3,-1.5) {};
\node [draw, circle] (6) at (-1,-1.5) {};
\node [draw, circle] (7) at (-2.67,0) {};
\node [draw, circle] (8) at (-1.33,0) {};
\node [draw, circle] (9) at (1.5, 1) {};
\node [draw, circle] (10) at (1.5, -1) {};
\draw [thick] (1)--(2)--(3)--(4)--(5)--(6)--(1)--(9)--(10)--(1)--(8)--(7)--(4);
\draw [thick] (5)--(7)--(2);
\draw [thick] (6)--(8);
\node at (-2,-2){$G$};
\node [draw, circle, blue] (31) at (-4, 2.5) {};
\node [draw, circle, blue] (32) at (-3, 2.5) {};
\node [draw, circle, blue] (33) at (-2, 2.5) {};
\node [draw, circle, blue] (334) at (-1, 2.5) {};
\node [draw, circle, blue] (335) at (-1, 3.5) {};
\draw [thick, blue] (31)--(3)--(32);
\draw [thick, blue] (3)--(33)--(334)--(33)--(335);
\node [draw, circle, blue] (101) at (0.5,-1.5) {};
\draw [thick, blue] (101)--(10);
\node [draw, circle, blue] (41) at (-5,0.7) {};
\node [draw, circle, blue] (42) at (-5,-0.7) {};
\node [draw, circle, blue] (421) at (-5,-2) {};
\draw [thick, blue] (421)--(42)--(4)--(41);
\node [draw, circle, blue] (11) at (0.5,1.1){};
\node [draw, circle, blue] (111) at (0,1.8){};
\node [draw, circle, blue] (112) at (1,1.8){};
\draw [thick, blue] (1)--(11)--(111)--(11)--(112);

\begin{scope}[shift={(7,0)}]
\node [draw, circle] (1) at (0,0) {};
\node [draw, circle] (2) at (-1,1.5) {};
\node [draw, circle] (3) at (-3,1.5) {};
\node [draw, circle] (4) at (-4,0) {};
\node [draw, circle] (5) at (-3,-1.5) {};
\node [draw, circle] (6) at (-1,-1.5) {};
\node [draw, circle] (7) at (-2.67,0) {};
\node [draw, circle] (8) at (-1.33,0) {};
\node [draw, circle] (9) at (1.5, 1) {};
\node [draw, circle] (10) at (1.5, -1) {};
\draw [thick] (1)--(2)--(3)--(4)--(5)--(6)--(1)--(9)--(10)--(1)--(8)--(7)--(4);
\draw [thick] (5)--(7)--(2);
\draw [thick] (6)--(8);
\node at (-2,-2){$\Prune[G]$};
\end{scope}

\end{tikzpicture}

\caption{Illustrating the pruning operation: a graph $G$ on the left, and its 2-core $\Prune[G]$ on the right. The tree branches of $G$ (removed in the pruning) are represented in blue.\label{fig:pruning}}
\end{figure}
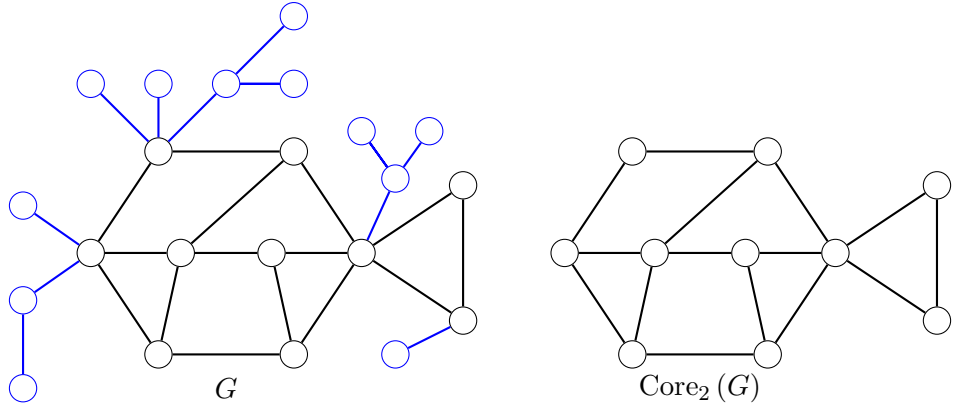

\begin{lemma}\label{lem:tree_pruning_tutte}
For any graph $G$, we have $a_G = a_{\Prune[G]}$.
\end{lemma}

\begin{proof}
Let $G$ be a graph with a vertex $v$ of degree $1$, and consider $e$ the unique edge incident to $v$. By definition, we have $G \setminus \{v\} = G \cdot e$, and $e$ is a bridge in $G$ (deleting it disconnects $v$ from the rest of the graph). Therefore we have $a_{G \setminus \{v\} } = a_{G \cdot e} = a_G$ by Equation~\eqref{eq:Tutte_poly} (bridge case). The result follows by iterating the above until no vertices of degree $1$ remain.
\end{proof}

Our next lemma concerns graphs whose pruned graph is a cycle. For $k \geq 1$, we say that a graph $G$ is a \emph{decorated $k$-cycle} if its pruned graph $\Prune[G]$ is the $k$-cycle $C_k$. Equivalently, $G$ contains a $k$-cycle, and deleting any edge in that $k$-cycle makes $G$ a tree.

\begin{lemma}\label{lem:tutte_dec_cycle}
If $G$ is a decorated $k$-cycle for some $k \geq 1$, then we have $a_G = k-1$.
\end{lemma}

\begin{proof}
If $G$ is a decorated $k$-cycle, then $\Prune[G]$ is the $k$-cycle $C_k$ by definition. Lemma~\ref{lem:tree_pruning_tutte} then yields that $a_G = a_{C_k}$. Applying the Deletion-contraction Equation~\eqref{eq:Tutte_poly} gives $a_{C_k} = a_{P_k} + a_{C_{k-1}}$ (where $P_k$ is the path graph on $k$ vertices). Since path graphs are trees, Lemma~\ref{lem:tutte_tree} implies that $a_{P_k} = 1$ for any $k \geq 1$, so that $a_{C_k} = k-1 + a_{C_1}$. But $C_1$ is the graph consisting of a single loop, so Lemma~\ref{lem:tutte_loop} implies that $a_{C_1} = 0$, and the result follows immediately.
\end{proof}

Our next lemma shows that adding an edge to a graph $G$ strictly increases its acyclic orientation number.

\begin{lemma}\label{lem:add_edge_tutte}
Let $G$ be a \emph{simple} graph, with a pair of \emph{distinct} vertices $v,w$ such that $(v,w)$ is \emph{not} an edge of $G$. Let $G' := G \cup \{(v,w)\}$ be the graph $G$ to which we add the edge $(v,w)$. Then we have $a_G < a_{G'}$.
\end{lemma}

If $G'$ is obtained from $G$ as above, we say that $G'$ is an \emph{edge addition} of $G$.

\begin{proof}
Let $G$, $e:=(v,w)$, and $G'$ be as in the statement of the lemma, i.e.\ $G = G' \setminus e$. By the Deletion-contraction Equation~\eqref{eq:Tutte_poly}, we have $a_{G'} = a_G + a_{G' \cdot e}$. Moreover, by construction, since $G$ is simple, the graph $G' \cdot e$ is loop-free. The result then follows from Lemma~\ref{lem:tutte_loop}. 
\end{proof}

Finally, the following result will be widely used throughout Section~\ref{sec:cnat_count} to obtain bounds on the acyclic orientation numbers of various graphs.

\begin{lemma}\label{lem:strict_subgraph}
Let $G$ be a simple graph, and $G'$ a subgraph of $G$. Then we have $a_{G'} \leq a_G$. Moreover, if the $2$-cores $\Prune[G]$ and $\Prune[G']$ are \emph{not} isomorphic, then we have $a_{G'} < a_G$.
\end{lemma}

\begin{proof}
Fix some vertex $s \in V(G')$. Consider the following algorithmic procedure.
\begin{enumerate}
\item Initialise $H = G'$.
\item While $V(H) \subsetneq V(G)$, do:
  \begin{enumerate}
  \item Choose a vertex $v \in V(G) \setminus V(H)$.
  \item Since $G$ is connected, there exists a path from $v$ to $s$ in $G$. Consider such a path, and let $w$ be the first vertex in $V(H)$ encountered on this path.
  \item Set $H$ to be the union of $H$ and the above path from $v$ to $w$.
  \end{enumerate}
\item Output $H :=G^0$.
\end{enumerate}
By construction, the above process outputs a spanning subgraph $G^0$ of $G$ such that $\Prune[G^0]$ and  $\Prune[G']$ are isomorphic (constructing $G^0$ from $G'$ only adds ``tree branches''). In particular by Lemma~\ref{lem:tree_pruning_tutte}, we have $a_{G^0} = a_{G'}$. Then the graph $G$ can be reached from $G^0$ through a (possibly empty) sequence of edge additions. Moreover, if $\Prune[G]$ and $\Prune[G']$ are not isomorphic, there must be at least one edge in $E(G)$ which is \emph{not} in $E(G^0)$ (otherwise we would have $G^0 = G$ and therefore $\Prune[G] = \Prune[G^0]$ and $\Prune[G']$ would be isomorphic), in which case this sequence must be non-empty. The results then follow from Lemma~\ref{lem:add_edge_tutte}.
\end{proof}

\subsection{Permutations, permutation patterns, and permutation graphs}\label{subsec:perms_prelims}

For $n \geq 1$, we let $S_n$ be the set of permutations of size $n$ (called $n$-permutations for short). We will usually use standard one-line notation for permutations. That is, a permutation $\pi \in S_n$ is a word $\pi = \pi_1 \cdots \pi_n$ on the alphabet $[n]$ such that every letter in $[n]$ appears exactly once in $\pi$. By convention, the unique $0$-permutation is simply the empty word. For $i,j \in [n], i \neq j$, we write $i \prec_{\pi} j$ if $i$ appears before (to the left of) $j$ in the permutation (word) $\pi$.

Another useful way of viewing permutations is their \emph{plot}. In this view, an $n$-permutation is a set of dots in an $n \times n$ grid such that every row and every column of the grid contains exactly one dot. For a permutation $\pi = \pi_1 \cdots \pi_n$, we construct its plot by putting dots in column $i$ and row $\pi_i$ for $i = 1, \ldots, n$. When connecting permutations to CNATs, it will be convenient to label columns from left-to-right and rows from top-to-bottom in this representation, as in Figure~\ref{fig:perm_graphical_repr} below.

\begin{figure}[ht]
\centering
\begin{tikzpicture}[scale=0.4]
\draw[step=2cm,thick] (0,0) grid (12,-12.01);
\tdot{1}{-9}{blue}
\tdot{3}{-11}{blue}
\tdot{5}{-1}{blue}
\tdot{7}{-3}{blue}
\tdot{9}{-7}{blue}
\tdot{11}{-5}{blue}
\foreach \x in {1,...,6}
	\node at (-1+2*\x,1) {$\x$};
\foreach \x in {1,...,6}
	\node at (-1,1-2*\x) {$\x$};
\end{tikzpicture}

\caption{Plot of the permutation $\pi = 561243$.\label{fig:perm_graphical_repr}}
\end{figure}
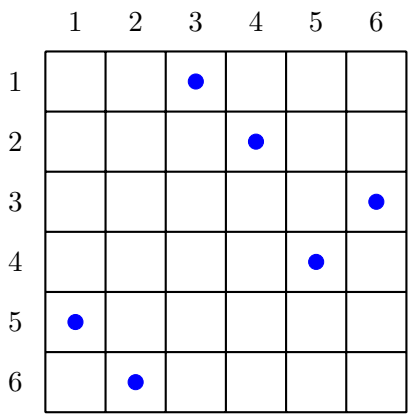

A permutation $\pi = \pi_1 \cdots \pi_n$ is said to be \emph{reducible} if there exists $1 \leq k < n$ such that $\pi_1 \cdots \pi_k$ is a $k$-permutation. A permutation is \emph{irreducible} if it is not reducible. Such an index $k$ is called a \emph{breakpoint} of the permutation $\pi$. For example, the permutation $\textbf{312}564$ is reducible, while $561243$ is irreducible.\footnote{
The terminology of \emph{(in)decomposable}, or \emph{sum (in)decomposable}, is also used sometimes in the literature instead of (ir)reducible.
}
A \emph{fixed point} of a permutation $\pi$ is an index $j \in [n]$ such that $\pi_j = j$. A \emph{left-to-right minimum} is a letter $\pi_i$ such that $\pi_j > \pi_i$ for all $j < i$. For example, the permutation $\textbf{5}6\textbf{1}243$ has two left-to-right minima $5$ and $1$.

Given a permutation $\pi$ and two letters $i,j \in [n]$, we say that $(i,j)$ is an \emph{inversion} of $\pi$ if $i > j$ and $i \prec_{\pi} j$. A \emph{descent} of $\pi$ is an inversion consisting of consecutive letters in the word, that is $i = \pi_k$ and $j = \pi_{k+1}$ for some $k \in [n-1]$. For example, the permutation $\pi = 561243$ has the following inversions (\textbf{descents} in bold): $(5,1), (5,2), (5,4), (5,3), \textbf{(6,1)}, (6,2), (6,4), (6,3), \textbf{(4,3)}$.  The \emph{permutation graph} of $\pi$ is the graph with vertex set $[n]$ and edge set the set of inversions of $\pi$. We denote this graph $G_{\pi}$. Figure~\ref{fig:perm_graph} shows the permutation graph of $\pi = 561243$. Note that the labels of the vertices are the row labels in the plot of $\pi$.

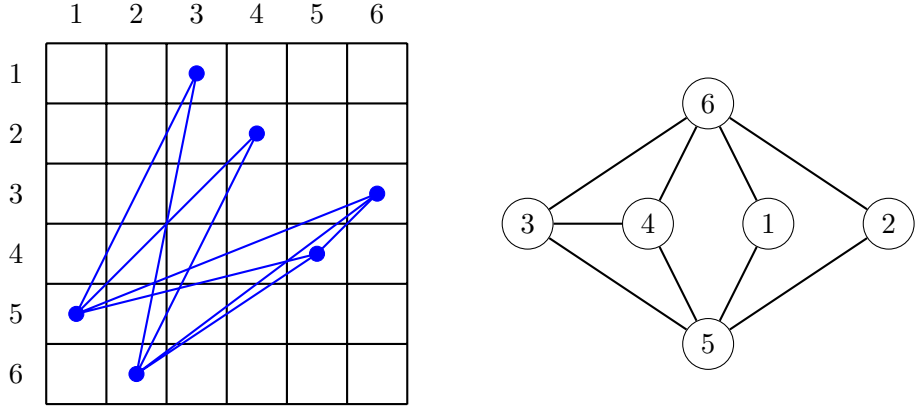
\begin{figure}[ht]
\centering
\begin{tikzpicture}[scale=0.4]
\draw[step=2cm,thick] (0,0) grid (12,-12.01);
\tdot{1}{-9}{blue}
\tdot{3}{-11}{blue}
\tdot{5}{-1}{blue}
\tdot{7}{-3}{blue}
\tdot{9}{-7}{blue}
\tdot{11}{-5}{blue}
\foreach \x in {1,...,6}
	\node at (-1+2*\x,1) {$\x$};
\foreach \x in {1,...,6}
	\node at (-1,1-2*\x) {$\x$};
\draw [thick, color=blue] (5,-1)--(1,-9)--(7,-3);
\draw [thick, color=blue] (9,-7)--(1,-9)--(11,-5);
\draw [thick, color=blue] (5,-1)--(3,-11)--(7,-3);
\draw [thick, color=blue] (9,-7)--(3,-11)--(11,-5);
\draw [thick, color=blue] (9,-7)--(11,-5);

\begin{scope}[shift={(22,-10)}, rotate=90]
\node [draw, circle] (3) at (4,6) {$3$};
\node [draw, circle] (4) at (4,2) {$4$};
\node [draw, circle] (1) at (4,-2) {$1$};
\node [draw, circle] (2) at (4,-6) {$2$};
\node [draw, circle] (5) at (0,0) {$5$};
\node [draw, circle] (6) at (8,0) {$6$};
\draw [thick] (5)--(3)--(6)--(4)--(5)--(1)--(6)--(2)--(5);
\draw [thick] (4)--(3);
\end{scope}
\end{tikzpicture}

\caption{The permutation graph of $\pi = 561243$. On the left, we draw the edges corresponding to inversions of $\pi$ on its plot. On the right, we re-draw the permutation graph in more readable form.\label{fig:perm_graph}}

\end{figure}

\begin{proposition}\label{pro:perm_graph_connected}
Let $\pi$ be a permutation. The permutation graph $G_{\pi}$ is connected if and only if the permutation $\pi$ is irreducible.
\end{proposition}

This is a classical result on permutation graphs, see e.g.~\cite[Lemma~3.2]{KR}. From now on, we assume that all permutations are irreducible unless explicitly stated otherwise. One permutation that will play an important role in our paper, particularly in Section~\ref{sec:upper-diag} and Theorem~\ref{thm:max_b(n,k)}, is the \emph{decreasing} permutation, defined by $\dec[n] := n(n-1)\cdots 1$ (for some $n \geq 1$). In $\dec[n]$ every pair is an inversion, so the corresponding graph is the complete graph $K_n$.

Let $\pi$ be an $n$-permutation, and $\tau$ a $k$-permutation for some $k \leq n$. We say that $\pi$ \emph{contains} the \emph{pattern} $\tau$ if there exist indices $i_1 < i_2 < \cdots < i_k$ such that $\pi_{i_1}, \pi_{i_2}, \ldots, \pi_{i_k}$ appear in the same relative order as $\tau$. In other words, if we delete all columns and rows other than $(i_1, \pi_{i_1}), \ldots, (i_k, \pi_{i_k})$ from the plot of $\pi$, we get the plot of $\tau$. In this case, we call $\pi_{i_1} \cdots \pi_{i_k}$ an \emph{occurrence} of the pattern $\tau$. If $\pi$ does not contain the pattern $\tau$, we say that $\pi$ \emph{avoids} $\tau$. For example, the permutation $\pi = 561243$ contains two occurrences of the pattern $321$, given by the bolded letters: $\textbf{5}612\textbf{43}$, $5\textbf{6}12\textbf{43}$. However, it avoids the pattern $4321$ since there is no sequence of $4$ letters in decreasing order.

The study of permutation patterns has been of wide interest in permutation research in recent decades, with much focus on the enumeration of permutations avoiding certain given patterns. For example, it is still an open problem to establish the growth rate of the set of $n$-permutations avoiding the pattern $1324$ (see~\cite{BBPP} and \cite{MN} for recent results in that direction). For more information on permutation patterns, we refer the interested reader to Kitaev's book~\cite{Kit} or Chapters~4 and 5 in the most recent edition of B\'ona's book~\cite{Bona}. For a detailed survey on permutation \emph{classes}, which are closely linked to the notion of pattern avoidance, we refer the reader to~\cite{Vat}.

Patterns of a permutation $\pi$ are linked to induced subgraphs of its permutation graph $G_{\pi}$ via the following observation. Let $\pi$ be an $n$-permutation, and $\tau$ a $k$-permutation, for some $k \leq n$. Suppose that $\pi_1, \ldots, \pi_k$ is an occurrence of the pattern $\tau$ in $\pi$. 
Then the induced subgraph $G\left[\{\pi_1, \ldots, \pi_k\} \right]$ is isomorphic to the permutation graph $G_{\tau}$. Conversely, if there exist vertices $\pi_1, \ldots, \pi_k$ such that $G \left[ \{ \pi_1, \ldots, \pi_k\} \right]$ is isomorphic to $G_{\tau}$, then $\pi_1, \ldots, \pi_k$ is an occurrence of some pattern $\tau'$ in $\pi$, such that the graphs $G_{\tau}$ and $G_{\tau'}$ are isomorphic. This yields the following.

\begin{proposition}\label{pro:perm_patterns_cycles}
Let $\pi$ be a permutation, and $G_{\pi}$ its corresponding permutation graph.
\begin{enumerate}
\item The graph $G_{\pi}$ induces a $3$-cycle if, and only if, the permutation $\pi$ contains the pattern $321$.
\item The graph $G_{\pi}$ induces a $4$-cycle if, and only if, the permutation $\pi$ contains the pattern $3412$.
\item The graph $G_{\pi}$ induces no cycle of length $5$ or more.
\end{enumerate}
\end{proposition}

Points~(1) and (2) follow from the above observation, combined with the fact that $3$- and $4$-cycles correspond uniquely to the permutations $321$ and $3412$ respectively. Point~(3) was shown in~\cite[Proposition~1.7]{BV}.

\subsection{Complete non-ambiguous trees and associated permutations}\label{subsec:cnats_prelims}

In this section we define (complete) non-ambiguous trees and introduce their associated permutations. Non-ambiguous trees were originally introduced by Aval \textit{et al.}\ in~\cite{ABBS} as a special case of the tree-like tableaux from~\cite{ABN}. 

\begin{definition}\label{def:nat}
A \emph{non-ambiguous tree} (NAT) is a filling of an $m \times n $ rectangular grid, where each cell is either dotted or not, satisfying the following conditions.
\begin{enumerate}
\item\label{cond:cnat_def_nonempty} Every row and every column contains at least one dotted cell.
\item\label{cond:cnat_def_uniqueparent} Aside from the top-left cell, every dotted cell has either a dotted cell above it in the same column, or a dotted cell to its left in the same row, but not both.
\end{enumerate}
\end{definition}

Note that the two conditions imply that the top-left cell must always be dotted. The use of the word \emph{tree} to describe these objects comes from the following observation. Given a NAT $T$, we connect every dot $d$ not in the top-left cell to its \emph{parent} dot $p(d)$, which is the dot immediately above it in its column or to its left in its row (by Condition~\eqref{cond:cnat_def_uniqueparent} of the above definition, exactly one of these must exist). This yields a binary tree, rooted at the top-left dot (see Figure~\ref{fig:nat_cnat}).

Following tree terminology, for a NAT $T$, we call the dot lying in the top-left cell the \emph{root dot}, or simply the \emph{root}, of $T$. Similarly, a \emph{leaf dot} is a dot with no dots below it in the same column or to its right in the same row. An \emph{internal dot} is a dot which is not a leaf dot (this includes the root, unless the NAT is a single dotted cell). Given a NAT, it will be convenient to label the columns $1, 2, \dots, n$ from left to right, and the rows $1, 2, \dots, m$ from top to bottom.

A NAT is said to be \emph{complete} if the underlying tree is complete, i.e.\ every internal dot has exactly two children. More formally, a complete non-ambiguous tree (CNAT) is a NAT in which every dot either has both a dot below it in the same column and a dot to its right in the same row, or neither of these. The \emph{size} of a CNAT is its number of leaf dots, or equivalently one more than its number of internal dots. 

\begin{figure}[ht]
\centering
\begin{tikzpicture}[scale=0.35]
\draw [step=2] (2,2) grid (12,-8);
\foreach \x in {1,...,5}
  \node at (1+2*\x, 3) {$\x$};
\foreach \y in {1,...,5}
  \node at (1, 3-2*\y) {$\y$};
\draw [thick] (5,-7)--(5,-1)--(11,-1);
\draw [thick] (3,-5)--(3,1)--(9,1);
\draw [thick] (7,1)--(7,-3);
\draw [thick] (3,-1)--(5,-1);
\tdot{3}{1}{black}
\tdot{7}{1}{black}
\tdot{3}{-1}{black}
\tdot{5}{-1}{black}
\tdot{3}{-5}{blue}
\tdot{5}{-7}{blue}
\tdot{7}{-3}{blue}
\tdot{9}{1}{blue}
\tdot{11}{-1}{blue}

\begin{scope}[shift={(17,2)}]
\draw [step=2] (2,0) grid (10,-10);
\foreach \x in {1,...,4}
  \node at (1+2*\x, 1) {$\x$};
\foreach \y in {1,...,5}
  \node at (1, 1-2*\y) {$\y$};
\node at (10.5,0.95) {$=n$};
\node at (1.05,-10.2) { \rotatebox{90}{$=$} };
\node at (1,-11.2) {$m$};
\draw [thick] (5,-9)--(5,-3)--(9,-3);
\draw [thick] (3,-1)--(7,-1);
\draw [thick] (3,-3)--(5,-3);
\draw [thick] (7,-1)--(7,-5);
\draw [thick] (3,-1)--(3,-7);
\tdot{3}{-1}{black}
\tdot{7}{-1}{red}
\tdot{3}{-3}{black}
\tdot{5}{-3}{black}
\tdot{3}{-7}{blue}
\tdot{5}{-9}{blue}
\tdot{7}{-5}{blue}
\tdot{9}{-3}{blue}
\end{scope}
\end{tikzpicture}

\caption{Two examples of NATs. Leaf dots are represented in blue, and internal dots in black. The NAT on the left is complete, while the one on the right is not (the red dot has only one child).\label{fig:nat_cnat}}
\end{figure}
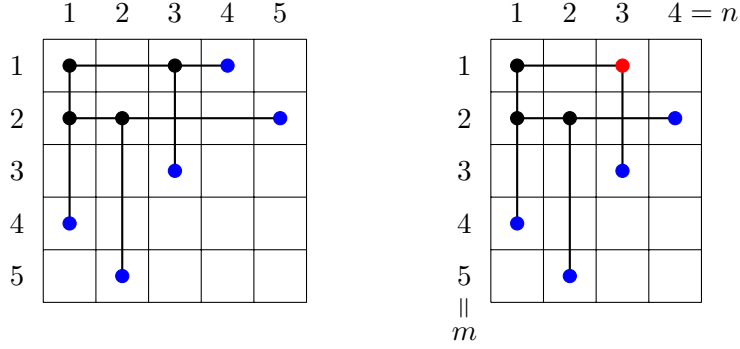

It is straightforward to check that in a CNAT, every row and every column must have exactly one leaf dot (see e.g.~\cite[Section~4.3]{DSSS2} or \cite[Section~2.1]{CO}). The leaf dot in a column, resp.\ row, is simply its bottom-most, resp.\ right-most, dot. As such, the set of leaf dots of a CNAT $T$ of size $n$ forms the plot of an $n$-permutation $\pi$, which must be irreducible (see e.g.~\cite[Theorem~3.3]{CO}). We say that $\pi$ is the permutation \emph{associated} with the CNAT $T$. For example, the CNAT on the left of Figure~\ref{fig:nat_cnat} has associated permutation $\pi = 45312$.  
We say that a row or column of a CNAT is \emph{empty} if it contains no dots other than its leaf dot.

\begin{definition}\label{def:cnat_number}
For a permutation $\pi$, we define the set $\CNAT[\pi]$ to be the set of CNATs whose associated permutation is $\pi$, and $\cnat[\pi] := \vert \CNAT[\pi] \vert$ to be the number of such CNATs.
\end{definition}

\subsection{Connecting CNATs to the acyclic orientation numbers}\label{subsec:cnats_tutte_prelims}

We are now ready to state the connection between CNATs and the acyclic orientation number. This is essentially a re-phrasing of~\cite[Corollary~33]{DSSS2} if one uses the bijection betwen minimal recurrent configurations of the Abelian sandpile model and $s$-rooted acyclic orientations of the underlying graph (see~\cite[Corollary~3]{Sch}).

\begin{theorem}\label{thm:cnat_tutte}
Let $\pi \in S_n$ be an $n$-permutation for some $n \geq 1$, and $G_{\pi}$ the corresponding permutation graph. Then we have
$$ \cnat[\pi] = a_{G_\pi}.$$
\end{theorem}


\section{Upper-diagonal CNATs and permutations}\label{sec:upper-diag}

In~\cite{CO}, the authors showed that for the decreasing permutation $\dec[n] = n(n-1)\cdots 1$, we have $\cnat[{\dec[n]}] = (n-1)!$. Their proof is via so-called tiered trees, with the authors exhibiting a bijection between upper-diagonal fully tiered trees on $n$ vertices and permutations of $[n-1]$. In this section, we give a new proof of this result by exhibiting a direct bijection between the set $\CNAT[{\dec[n]}]$ and the set of permutations of $[n-1]$. Our bijection also maps certain statistics of CNATs to statistics on the corresponding permutation. Following notation from~\cite{CO}, we call a CNAT \emph{upper-diagonal} if its associated permutation is a decreasing permutation.

\begin{definition}\label{def:labelled_cnat}
A \emph{labelled CNAT} is a pair 
$\T = (T, \lambda)$
 where $T$ is an upper-diagonal CNAT of size $(n+1)$ for some $n \geq 0$ and $\lambda = \{ \lambda_1, \lambda_2, \ldots, \lambda_n \} \subset \N^n$ is a set of $n$ distinct natural numbers. We call $\lambda$ the \emph{labels} of $\T$, and say that $\T$ is a $\lambda$-labelled CNAT.
\end{definition}

In the above, we think of $\lambda$ as labelling the columns of the CNAT $T$ other than the right-most column, as in Figures~\ref{fig:top_decomp_cnat} and \ref{fig:top_del_cnat} below (the labels are written on top of the columns). With a slight abuse of notation, we assume that $\lambda_1 < \lambda_2 < \cdots < \lambda_n$, and use $\lambda$ to denote both the set of labels and the ordered tuple $(\lambda_1, \ldots, \lambda_n)$. Note that the number of labels is one less than the size of the underlying CNAT.

Now suppose that $\T = (T, \lambda)$ is a $\lambda$-labelled CNAT, such that the underlying CNAT $T$ has at least two internal dots in the top row. Let $\lambda_r$ be the label of the right-most of these, i.e.\ the right-most internal dot in the top row of $T$ is in the $r$-th column. The \emph{top-row decomposition} of $\T$ is the pair $\left( \T^{\ell}, \T^r \right)$ defined as follows:
 \begin{itemize}
 \item $\T^r := (T^r, \lambda^r)$, where $T^r$ is the sub-tree of $T$ whose root is the dot in cell $(r,1)$ (the right-most internal dot in the top row of $T$). In other words, $T^r$ is obtained from $T$ by keeping exactly the rows and columns containing dots whose path to the root in $T$ goes through the dot in cell $(r,1)$. Then $\lambda^r$ is the set of labels whose columns have at least one dot in $T^r$.
 \item $\T^{\ell} := (T^{\ell}, \lambda^{\ell})$, where $T^{\ell}$ is obtained from $T$ by replacing the internal dot in cell $(r,1)$ by a leaf dot and moving it to the right-most column, and $\lambda^{\ell}$ is again the set of labels whose columns have at least one dot in $T^{\ell}$ (excluding the new leaf dot).
 \end{itemize}
We will sometimes refer to $\T^r$, resp.\ $\T^{\ell}$, as the right, resp.\ left, subtree of $\T$. We write $\T = \T^{\ell} \oplus \T^r$ for the top-row decomposition of a labelled CNAT $\T$. Note that this partitions the label set $\lambda$ into two (disjoint) subsets $\lambda^{\ell}$ and $\lambda^r$. This operation of a labelled CNAT is illustrated in Figure~\ref{fig:top_decomp_cnat}. 

\begin{remark}\label{rem:top_row_decomp_invert} The top-row decomposition operation is invertible in the following sense. Given two labelled CNATs $\T^{\ell} = (T^{\ell}, \lambda^{\ell})$ and $\T^r = (T^r, \lambda^r)$ with disjoint label sets such that the minimum label $\lambda^r_1$ of $\T^r$ is strictly greater than the labels of all (internal) top row dots in $\T^{\ell}$, then there exists a unique labelled CNAT $\T$ such that $\T = \T^{\ell} \oplus \T^r$. Here $\T$ is obtained by ``gluing'' the tree $T^r$ onto the right-most leaf of $T^{\ell}$ (in the top row), and shifting the columns of the glued subtree in such a way that the labelling of the tree obtained remains in increasing order (here we use the fact that the leaves are all on the diagonal to find the unique way of interleaving the columns).
\end{remark}

\begin{figure}[ht]
\centering
\begin{tikzpicture}[scale=0.3]
\node at (-3, 2.5) {$2$};
\node at (-1, 2.5) {$4$};
\node at (1, 2.5) {$5$};
\node at (3, 2.5) {$7$};
\node at (5, 2.5) {$9$};
\node at (7, 2.5) {$10$};
\node at (9, 2.5) {$13$};
\node at (11, 2.5) {$14$};
\node at (13, 2.5) {$16$};
\draw [thick] (15,1)--(1,1)--(1,-13);
\draw [thick] (5,-9)--(5,1);
\draw [thick] (1,-5)--(9,-5);
\draw [thick] (11,-3)--(5,-3);
\draw [thick] (7,-3)--(7,-7);
\draw [thick] (1,-11)--(3,-11);
\draw [thick] (-3,-17)--(-3,1)--(1,1);
\draw [thick] (-3,-1)--(13,-1);
\draw [thick] (-1,-15)--(-1,-1);
\tdot{-3}{1}{red}
\tdot{-3}{-1}{red}
\tdot{-1}{-1}{red}
\tdot{1}{1}{red}
\tdot{1}{-5}{red}
\tdot{7}{-3}{red}
\tdot{1}{-11}{red}
\tdot{5}{1}{red}
\tdot{5}{-3}{red}
\foreach \y in {1,...,10}
	\draw [fill=black,color=blue] (-5+2*\y,-19+2*\y) circle [radius=0.25];

\node at (18,-0.5) {$=$};

\begin{scope}[shift={(24,0)}]
\node at (-3, 2.5) {$2$};
\node at (-1, 2.5) {$4$};
\node at (1, 2.5) {$5$};
\node at (3, 2.5) {$7$};
\node at (5, 2.5) {$13$};
\node at (7, 2.5) {$16$};
\draw [thick] (-3,-11)--(-3,-1)--(-1,-1)--(-1,-9);
\draw [thick] (-3,-1)--(-3,1)--(1,1)--(1,-7);
\draw [thick] (3,-5)--(1,-5)--(1,-3)--(5,-3);
\draw [thick] (-1,-1)--(7,-1);
\draw [thick] (-1,1)--(9,1);
\tdot{-3}{1}{red}
\tdot{-3}{-1}{red}
\tdot{-1}{-1}{red}
\tdot{1}{1}{red}
\tdot{1}{-3}{red}
\tdot{1}{-5}{red}
\foreach \y in {1,...,7}
	\draw [fill=black,color=blue] (-5+2*\y,-13+2*\y) circle [radius=0.25];
\end{scope}

\node at (35,-0.5) {$\oplus$};

\begin{scope}[shift={(32,0)}]
\node at (5, 2.5) {$9$};
\node at (7, 2.5) {$10$};
\node at (9, 2.5) {$14$};
\draw [thick] (5,-5)--(5,1)--(11,1);
\draw [thick] (7,-3)--(7,-1)--(9,-1);
\draw [thick] (5,-1)--(7,-1);
\tdot{5}{1}{red}
\tdot{5}{-1}{red}
\tdot{7}{-1}{red}
\foreach \y in {1,...,4}
	\draw [fill=black,color=blue] (3+2*\y,-7+2*\y) circle [radius=0.25];
\end{scope}
\end{tikzpicture}
\caption{Illustrating the top-row decomposition of a labelled CNAT.\label{fig:top_decomp_cnat}}
\end{figure}
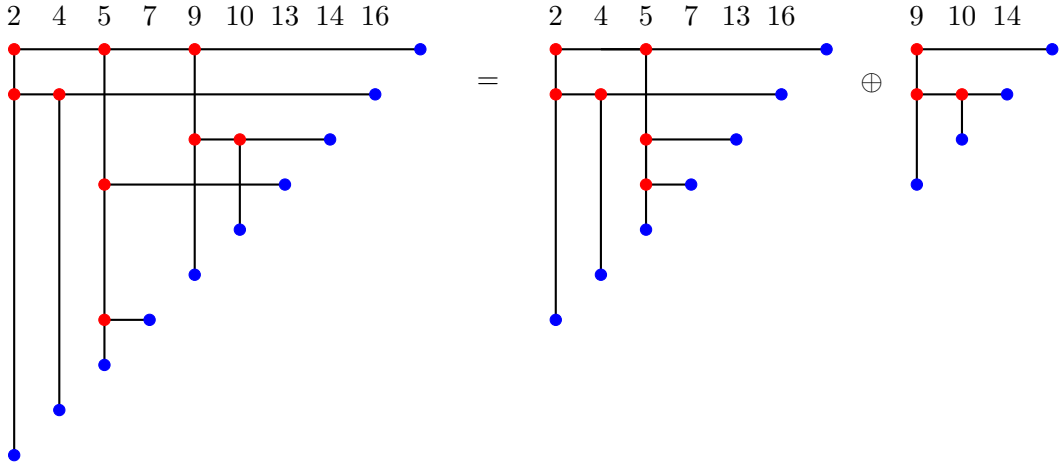

Now suppose that we have a labelled CNAT $\T = (T, \lambda)$ such that the only internal dot in the top row of $T$ is the root of the tree. Such a tree must necessarily have a dot in the second row of the first column, which we call $\rho'$. We define the \emph{top-row deletion} of $\T$ as the labelled tree $\T' := (T', \lambda')$ where $T'$ is the CNAT $T$ with top row deleted (i.e.\ the subtree rooted at $\rho'$), and $\lambda' := \lambda \setminus \{\lambda_1\}$ is the set of labels $\lambda$ with its minimum element removed. This operation is shown in Figure~\ref{fig:top_del_cnat}, with the labelled CNAT on the left, and its top-row deleted labelled CNAT on the right.

\begin{remark}\label{rem:top_row_del_invert}
The top-row deletion operation is not completely invertible, since we lose the information of the minimum label of $\T$. However, this is the only information that is lost. That is, given a labelled CNAT $\T' = (T', \lambda')$, and a label $k < \min \lambda'$, there exists a unique labelled CNAT $\T = (T, \lambda:=\lambda' \cup \{k\})$ whose top-row deletion is $\T'$.
\end{remark}

\begin{figure}[ht]
\centering
\begin{tikzpicture}[scale=0.3]

\node at (3, 2.5) {$1$};
\node at (5, 2.5) {$2$};
\node at (7, 2.5) {$5$};
\node at (9, 2.5) {$7$};
\draw [thick] (5,-5)--(5,-1)--(9,-1);
\draw [thick] (3,-7)--(3,1)--(11,1);
\draw [thick] (3,-3)--(7,-3);
\draw [thick] (3,-1)--(5,-1);
\tdot{3}{1}{red}
\tdot{3}{-1}{red}
\tdot{5}{-1}{red}
\tdot{3}{-3}{red}
\foreach \y in {1,...,5}
	\draw [fill=black,color=blue] (1+2*\y,-9+2*\y) circle [radius=0.25];
	
\node at (14,-0.5) {$\longrightarrow$};

\begin{scope}[shift={(15,2)}]
\node at (3, 0.5) {$2$};
\node at (5, 0.5) {$5$};
\node at (7, 0.5) {$7$};
\draw [thick] (5,-5)--(5,-1)--(9,-1);
\draw [thick] (3,-7)--(3,-1);
\draw [thick] (3,-3)--(7,-3);
\draw [thick] (3,-1)--(5,-1);
\tdot{3}{-1}{red}
\tdot{5}{-1}{red}
\tdot{3}{-3}{red}
\foreach \y in {1,...,4}
	\draw [fill=black,color=blue] (1+2*\y,-9+2*\y) circle [radius=0.25];
\end{scope}
	
\end{tikzpicture}

\caption{Illustrating the top-row deletion of a labelled CNAT.\label{fig:top_del_cnat}}
\end{figure}
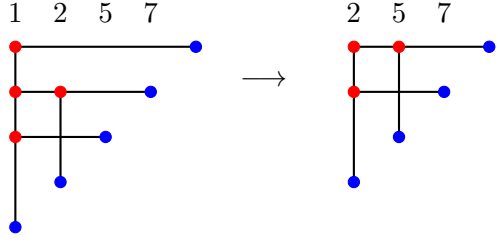

We are now equipped to define the desired bijection from upper-diagonal CNATs to permutations. In fact, given any label set $\lambda$, we define a bijection from $\lambda$-labelled CNATs to set-permutations of the label set $\lambda$ ($\lambda$-permutations for short). Since upper-diagonal CNATs of size $n$ can be viewed as $[n-1]$-labelled CNATs, this gives the desired bijection.

Fix a label set $\lambda = (\lambda_1, \ldots, \lambda_n)$, and a $\lambda$-labelled CNAT $\T = (T, \lambda)$. We define a set-permutation $\Psi(\T)$ of the label set $\lambda$ recursively as follows.
\begin{itemize}
\item If $T$ is the CNAT reduced to a single root dot, i.e.\ is the CNAT of size $1$ (in which case $\lambda = \emptyset$), we define $\Psi(\T)$ to be the empty word (base case).
\item If the only internal dot in the top row of $T$ is the root, we define recursively $\Psi(\T) := \lambda_1 \concat \Psi(\T')$ where $\T'$ is the top-row deletion of $\T$ (here the $\cdot$ operation denotes concatenation).
\item Otherwise, if there is more than one internal dot in the top row of $T$, we define recursively $\Psi(\T) := \Psi \left( \T^r \right) \concat \Psi \left( \T^{\ell} \right)$, where $\T = \T^{\ell} \oplus \T^r$ is the top-row decomposition of $\T$.
\end{itemize}

\begin{remark}\label{rem:alternative_tableaux_decomposition}
A similar decomposition was given for so-called \emph{alternative tableaux} by Nadeau~\cite[Section~2]{Nad}. Alternative tableaux are fillings of Ferrers diagrams with left- and up-arrows, and are known to be in bijection with tree-like tableaux~\cite{ABN}, which are themselves in bijection with upper-diagonal CNATs~\cite{CO}.

More precisely, Nadeau's decomposition uses a concept of \emph{free columns/rows} for alternative tableaux. In the (labelled) CNAT setting, free columns correspond to the labels of dots in the top row, and Nadeau's decomposition along free columns is essentially equivalent to our top-row decomposition operation. On the other hand, our top row deletion operation appears new, especially the relabelling that is required. In particular, this operation allows us to only act on the columns of the upper-diagonal CNAT, without needing to introduce equivalent decomposition operation(s) on its rows. Our construction also positions itself directly in the CNAT setting, and is very natural in terms of the underlying tree structure.
\end{remark}

\begin{example}\label{ex:bij_cnat_perm}
Consider the labelled CNAT $\T$ in the left of Figure~\ref{fig:top_decomp_cnat}. We construct the associated $\lambda$-permutation $\Psi(T)$ of the label set $\lambda = (2, 4, 5, 7, 9, 10, 13, 14, 16)$. The first step is the top-row decomposition of $\T$ as in the figure. We therefore have $\Psi(\T) = \Psi \left( \T^r \right) \concat \Psi \left( \T^{\ell} \right)$, where $\T^r$ is the CNAT with label set $(9, 10, 14)$ and $\T^{\ell}$ the CNAT with label set $(2, 4, 5, 7, 13, 16)$ in the right-hand side of the equality in Figure~\ref{fig:top_decomp_cnat}. Here $\T^r$ has a unique dot in the top row, so $\Psi\left( \T^r \right) = 9 \concat \Psi \big( \left( \T^r \right)' \big)$, where $\left( \T^r \right)'$ is the top-row deletion of $\T^r$. This is the CNAT with $2$ internal dots in the top row and label set $(10, 14)$. Applying top-row decomposition followed by top-row deletion and the base case to the two subtrees gives $\Psi \big( \left( \T^r \right)' \big) = 14 \concat1 0$, so that $\Psi \left( \T^r \right) = 9 \concat 14 \concat 10$.

We then compute $\Psi \left( \T^{\ell} \right)$. Again, we start by applying top-row decomposition. The right sub-tree $\left( \T^{\ell} \right)^{r}$ in this decomposition has all its internal dots in the left-most column. One can easily check through successive applications of top-row deletion that such a labelled CNAT maps to the increasing permutation of its label set, so that $\Psi \left( \left( \T^{\ell} \right)^{r} \right) = 5 \concat 7 \concat 13 $. Finally, applying top-row deletion followed by top-row decomposition on the left sub-tree $\left( \T^{\ell} \right)^{\ell}$ gives $\Psi \left( \left( \T^{\ell} \right)^{\ell} \right) = 2 \concat 16 \concat 4$. Bringing it all together, we get $\Psi \left( \T \right) = 9 \concat 14 \concat 10 \concat 5 \concat 7 \concat 13 \concat 2 \concat 16 \concat 4$. 
\end{example}

Note that in Example~\ref{ex:bij_cnat_perm} above, the left-to-right minima of the $\lambda$-permutation $\pi := \Psi(\T)$ are exactly the labels of the top-row dots of $\T$, i.e.\ $2, 5, 9$. Moreover, $\pi$ has $4 = 5-1$ descents, and there are $5$ empty rows, i.e.\ rows whose only dot is their leaf dot, in the underlying CNAT $T$ (the rows whose leaves are in columns labelled $2$, $4$, $5$, $9$, and $10$). It turns out that both of these observations are true in general, which is the main result of this section.

\begin{theorem}\label{thm:bij_cnat_perm}
For any label set $\lambda$, the map $\Psi : \T \mapsto \Psi(T)$ is a bijection from the set of $\lambda$-labelled CNATs to the set of $\lambda$-permutations. Moreover, this bijection maps column labels of internal top row dots in $\T$ to left-to-right minima in the $\lambda$-permutation $\Psi(\T)$. Finally, the number of empty rows in the underlying CNAT $T$ is equal to one plus the number of descents of $\Psi(\T)$.
\end{theorem}

In particular, if the label set is $\lambda = [n-1]$ for some $n \geq 1$, this gives us a bijection between upper-diagonal CNATs of size $n$ and $(n-1)$-permutations, providing a direct proof of Theorem~4.12 in~\cite{CO}, as announced.

\begin{remark}\label{rem:bij_cnat_perm_asm}
Recall that upper-diagonal CNATs are those whose associated permutation is $\dec[n] = n(n-1)\cdots 1$. The corresponding permutation graph $G_{\dec[n]}$ is simply the complete graph on $n$ vertices $K_n$. It is straightforward to see that acyclic orientations of $K_n$ are simply linear orders of its vertex set, i.e.\ $n$-permutations, and fixing the unique sink of these yields $(n-1)$-permutations, so that $a_{K_n} = (n-1)!$ by Proposition~\ref{pro:tutte_acyc_or}. 

On the other hand, the proof of Theorem~\ref{thm:cnat_tutte} in~\cite{DSSS2} goes through an alternate interpretation of the acyclic orientation number $a_{G_{\pi}}$ in terms of spanning trees of $G_{\pi}$ with no external activity (these are shown to correspond bijectively to CNATs with associated permutation $\pi$). In this context, Theorem~\ref{thm:bij_cnat_perm} can be viewed as a bijection between spanning trees with no external activity and sink-rooted acyclic orientations in the special case of complete graphs. Such bijections do in fact exist in the general case (see e.g.~\cite{Ber} or \cite{GesSag}). The novelty of our bijection is placing this work directly in the CNAT setting, as well as the preservation of relevant statistics.
\end{remark}

\begin{proof}
Since top-row decomposition partitions the label set $\lambda$ into two (smaller) label sets, and top-row deletion removes the minimum label $\lambda_1$, we have that $\Psi(\T)$ is a $\lambda$-permutation by a straightforward induction on the size of $\lambda$ using the recursive definition. 

Let us now show that if there is a dot in the top row in the column labelled $\lambda_i$ of $\T$, then $\lambda_i$ is a left-to-right minimum of the permutation $\Psi(\T)$. Again, we proceed by induction on $n$, the size of the label set $\lambda$. For $n=0$ the result is trivial. Now fix some $n > 0$, and suppose we have shown this for all label sets $\lambda'$ of size at most $n-1$. Fix a label set $\lambda = (\lambda_1, \ldots, \lambda_n)$ of size $n$, and a $\lambda$-labelled CNAT $\T$.

Firstly, consider the case where $\T$ has a unique internal dot in the top row, which is necessarily the root, in column labelled $\lambda_1 = \min \lambda$. By construction $\lambda_1$ is the first element in the $\lambda$-permutation $\Psi(\T)$, and is therefore its unique left-to-right minimum, as desired.

It thus remains to consider the case where $\T$ has at least two internal dots in the top row. In this case write the top-row decomposition $\T = \T^{\ell} \oplus \T^r$ of $\T$.  By construction the labels of top-row internal dots of $\T$ are exactly those of top-row internal dots of $\T^{\ell}$, plus the label $\lambda^r_1$ of the right-most (internal) dot in that row (the root of $T^r$). By the induction hypothesis, $\Psi(\T^{\ell})$ is a permutation of $\lambda^{\ell}$ whose left-to-right minima are the column labels of the top-row dots of $\T^{\ell}$. In particular, this implies that the first letter in $\Psi(\T^{\ell})$ is strictly less than $\lambda^r_1$. Moreover, by the above argument $\Psi(\T^r)$ is a permutation of $\lambda^r$ starting with $\lambda^r_1$ (the minimal element of $\lambda^r$). This implies that the left-to-right minima in the permutation $\Psi(\T) = \Psi(\T^r) \concat \Psi(\T^{\ell})$ are exactly $\lambda^r_1$ and the left-to-right minima of $\Psi(\T^{\ell})$, and the result follows by applying the induction hypothesis to $\Psi(\T^{\ell})$.

To show that $\Psi$ is a bijection, we construct its inverse recursively, and show it maps left-to-right minima to top-row internal dots in the labelled CNAT. Let $\pi = \pi_1\cdots \pi_n$ be a permutation of a label set $\lambda = (\lambda_1, \ldots, \lambda_n)$. We define a $\lambda$-labelled CNAT $\T := \Phi(\pi)$ as follows.
\begin{itemize}[topsep=2pt]
\item If $\lambda = \emptyset$ ($\pi$ is the empty word in this case), we set $T$ to be the CNAT consisting of a single dot, and $\T = (T, \emptyset)$ to be the labelled CNAT (base case).
\item If $\pi_1 = \lambda_1$ (the first element of $\pi$ is the minimum label in $\lambda$), then we first calculate recursively $\T' = \Phi(\pi_2\cdots \pi_n)$, and then let $\T$ be the unique $\lambda$-labelled CNAT whose top-row deletion is $\T'$ (see Remark~\ref{rem:top_row_del_invert}). By construction the CNAT has a unique internal dot in the top-row, which is in column labelled $\lambda_1$, giving the desired result.
\item Otherwise, define $k = \min \{i \geq 2; \pi_i < \pi_1 \}$, and decompose $\pi = (\pi_1 \cdots \pi_{k-1})(\pi_k \cdots \pi_n) := \pi^r \pi^{\ell}$. In other words, the first element of $\pi^{\ell}$ is the second left-to-right minimum in $\pi$. Then calculate $\T^{\ell} = \Phi(\pi^{\ell})$ and $\T^r = \Phi(\pi^r)$ recursively. By construction, the permutation $\pi^r$ begins with its minimum label, so the top-row of $\T^r$ is empty except for the leaf and root dots from the above case. Moreover, the minimum label of $\T^r$ is $\pi_1$, which by definition is greater than all left-to-right minima in $\pi^{\ell}$. By the induction hypothesis, this means exactly that the minimum label of $\T^r$ is greater than the labels of all the top-row internal dots in $\T^{\ell}$. Remark~\ref{rem:top_row_decomp_invert} then states that there exists a unique labelled CNAT $\T$ such that $\T = \T^{\ell} \oplus \T^r$, and we set $\Phi(\pi) = \T$. By construction, the column labels of the top-row dots of $\T$ are those of $\T^{\ell}$ plus that of the root of $\T^r$, which by the induction hypothesis are exactly the left-to-right minima of $\pi$, as desired.
\end{itemize}
That $\Psi$ and $\Phi$ are inverse of each other is straightforward from the construction, which proves the first part of the theorem.

It remains to show that the number of empty rows in $\T$ is one plus the number of descents of $\Psi(\T)$. For a labelled CNAT $\T$, resp.\ a permutation $\pi$, let $\ER[\T]$, resp.\ $\des[\pi]$, be its number of empty rows, resp.\ descents. Again, we proceed by induction on the size $n$ of the label set $\lambda$. The base case $n=0$ is trivial. For $n \geq 1$, suppose we have shown that $\ER[\T'] = 1 + \des[\Psi(\T')]$ for all $\lambda'$ of size at most $n-1$ and $\lambda'$-labelled CNATs $\T'$.

Fix some label set $\lambda = (\lambda_1, \ldots, \lambda_n)$ and $\lambda$-labelled CNAT $\T$. 
\begin{itemize}[topsep=2pt]
  \item If $\T$ has a unique internal dot in its top-row, let $\T'$ be the top-row deletion of $\T$. Then $\ER[\T] = \ER[\T'] = 1 + \des[\Psi(\T')]$ by induction. But by construction $\Psi(\T) = \lambda_1 \concat \Psi(\T')$ with $\lambda_1 = \min \lambda$, which immediately implies $\des[\Psi(\T)] = \des[\Psi(\T')]$, and the desired result follows.
  \item Otherwise, write $\T = \T^{\ell} \oplus \T^r$ for the top-row decomposition of $\T$. By construction we have $\Psi(\T) = \Psi(\T^r) \concat \Psi(\T^{\ell})$. Moreover, we know from the previous part of the proof that the first element of $\Psi(\T^r)$, which is also its smallest element, is strictly greater than all the left-to-right minima of $\Psi(\T^{\ell})$, and thus in particular than its first element. This implies that there is a descent between the last element of $\Psi(\T^r)$ and the first element of $\Psi(\T^{\ell})$, so that $\des[\Psi(\T)] = \des[\Psi(\T^r)] + 1 + \des[\Psi(\T^{\ell})]$. Applying the induction hypothesis to $\T^r$ and $\T^{\ell}$, this yields:
   \begin{align*}
   \ER[\T] & = \ER[\T^r] + \ER[\T^{\ell}] \\
    & = 1 + \des[\Psi(\T^r)] + 1 + \des[\Psi(\T^{\ell})] \\
    & = 1 + \des[\Psi(\T)],
   \end{align*}  
  as desired. This completes the proof of the theorem.
\end{itemize}
\end{proof}


\section{Counting CNATs according to their associated permutations}\label{sec:cnat_count}

In~\cite{CO} the authors were interested in the numbers $\cnat[\pi]$ enumerated in Theorem~\ref{thm:cnat_tutte}. More specifically, they were interested in how many permutations $\pi$ had the same fixed number of associated CNATs. For $n,k \geq 1$, we define
$ B(n, k) := \{\pi \in S_n;\, \cnat[\pi] = k \} $
to be the set of $n$-permutations associated with exactly $k$ CNATs, and $b(n,k) := \vert B(n,k) \vert$ to be the number of such permutations. For brevity we will simply say that a permutation $\pi \in B(n,k)$ \emph{has} $k$ CNATs, and refer to elements of $B(n,k)$ as permutations \emph{with} (exactly) $k$ CNATs. The main goal of this section is to prove a number of conjectures on the enumeration sequence $\big( b(n,k) \big)_{n, k \geq 1}$ from~\cite{CO}.

\subsection{Linking permutations with one and two CNATs}\label{subsec:B(n,1)_B(n+1,2)}

In this part, we establish a bijection between \emph{marked} permutations with a single CNAT, and permutations with $2$ CNATs (Theorem~\ref{thm:Insert_B(n,1)toB(n+1,2)}). We begin by describing the set $B(n,1)$ of permutations with a single CNAT. In~\cite{CO}, the authors gave a characterisation of this set in terms of so-called \emph{L-subsets}. We first reformulate their characterisation in terms of \emph{quadrants}. For an (irreducible) $n$-permutation $\pi$, we let $\Pi = \{ (i, \pi_i); i \in [n] \}$ be the points in the plot of $\pi$. Given an index $k \in \{2, \ldots, n\}$, we partition $\Pi$ into four quadrants, called $k$-quadrants when the index $k$ needs to be explicit, as follows. 
\begin{itemize}
\item The upper-left quadrant is $\Pi_{<k, <k} := \Pi \cap \Big( [1, k-1] \times [1, k-1] \Big)$.
\item The lower-left quadrant is $\Pi_{<k, \geq k} := \Pi \cap \Big( [1, k-1] \times [k, n] \Big)$.
\item The upper-right quadrant is $\Pi_{\geq k, <k} := \Pi \cap \Big( [k, n] \times [1, k-1] \Big)$.
\item The lower-right quadrant is $\Pi_{\geq k, \geq k} := \Pi \cap \Big( [k, n] \times [k, n] \Big)$.
\end{itemize}
This partition is illustrated in Figure~\ref{fig:quad_partition} below, with $\pi = 561243$ and $k=4$.

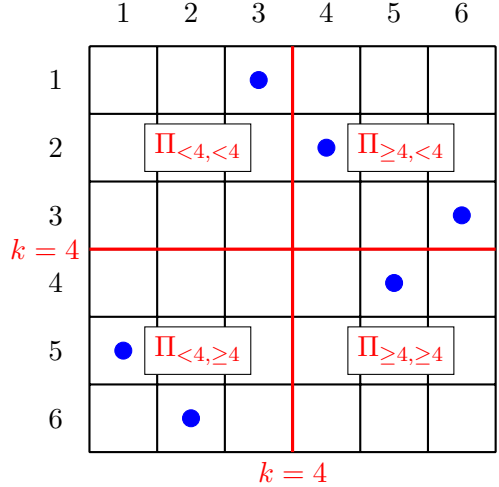
\begin{figure}[ht]
\centering
\begin{tikzpicture}[scale=0.45]
\draw[step=2cm,thick] (0,0) grid (12,-12.01);
\tdot{1}{-9}{blue}
\tdot{3}{-11}{blue}
\tdot{5}{-1}{blue}
\tdot{7}{-3}{blue}
\tdot{9}{-7}{blue}
\tdot{11}{-5}{blue}
\foreach \x in {1,...,6}
	\node at (-1+2*\x,1) {$\x$};
\foreach \x in {1,...,6}
	\node at (-1,1-2*\x) {$\x$};
\node [left] at (0,-6) {$\textcolor{red}{k=4}$};
\node [below] at (6,-12) {$\textcolor{red}{k=4}$};
\draw [red, very thick] (0,-6)--(12,-6);
\draw [red, very thick] (6,0)--(6,-12);
\node [draw, fill=white] at (3.2,-3) {$\textcolor{red}{\Pi_{<4, <4}}$};
\node [draw, fill=white] at (3.2,-9) {$\textcolor{red}{\Pi_{<4, \geq 4}}$};
\node [draw, fill=white] at (9.2,-3) {$\textcolor{red}{\Pi_{\geq 4, <4}}$};
\node [draw, fill=white] at (9.2,-9) {$\textcolor{red}{\Pi_{\geq 4, \geq 4}}$};
\end{tikzpicture}
\caption{The partition of a permutation's plot into four quadrants.\label{fig:quad_partition}}
\end{figure}

Note that since $\pi$ is a permutation, the first $k-1$ columns or rows of its plot must have $k-1$ dots in total. Therefore we have 
\beq\label{eq:quad_eq}
\left\vert \Pi_{<k, < k} \right\vert + \left\vert \Pi_{<k, \geq k} \right\vert = \left\vert \Pi_{<k, < k} \right\vert + \left\vert \Pi_{\geq k, < k} \right\vert = k-1,
\eeq 
so that $\left\vert \Pi_{<k, \geq k} \right\vert = \left\vert \Pi_{ \geq k, < k} \right\vert$. Moreover, this number must be strictly positive, since otherwise $\pi$ would induce a permutation on $[k-1]$, and therefore be reducible. In words, given an (irreducible) $n$-permutation $\pi$, and an index $k \in \{2, \ldots, n\}$, the lower-left and upper-right quadrants are non-empty, and have the same number of dots.

\begin{definition}\label{def:quad_cond}
Let $n \geq 2$. We say that an $n$-permutation $\pi$ satisfies the \emph{quadrant condition} if
\beq\label{eq:quad_cond}
\forall k \in \{2, \ldots, n\}, \ \left\vert \Pi_{<k, \geq k} \right\vert = \left\vert \Pi_{ \geq k, < k} \right\vert = 1.
\eeq
That is, a permutation satisfies the quadrant condition when every lower-left quadrant (equivalently upper-right quadrant) has exactly one dot.
\end{definition}

We now state our characterisation of the set $B(n,1)$.

\begin{proposition}\label{pro:B(n,1)_quadrants}
Let $\pi$ be a permutation. Then $\pi$ has a unique CNAT if, and only if, $\pi$ satisfies the quadrant condition and has no fixed point.
\end{proposition}

\begin{proof}
Let $\pi$ be an $n$-permutation. For $j \in [n]$, we define the $j$-th \emph{L-subset} $L_j$ of a permutation $\pi$ by $L_j := \Pi_{<j+1, <j+1} \setminus \Pi_{<j, <j}$. Therorem~3.11 in~\cite{CO} states that a permutation has a unique CNAT if, and only if, $\pi$ has no fixed point, $\vert L_1 \vert = 0$, $\vert L_n \vert = 2$, and $\vert L_j \vert = 1$ for all $j \in \{2, \ldots, n-1 \}$. We show that this is equivalent to $\pi$ having no fixed point and satisfying the quadrant condition.

Suppose first that $\pi$ satisfies the above conditions on its L-subsets, and fix some $k \in \{2,\ldots,n\}$. Then we have $ \left\vert \Pi_{<k,<k} \right\vert = \sum\limits_{j=1}^{k-1} \vert L_j \vert = k-2$. By Equation~\eqref{eq:quad_eq} this implies that $\left\vert \Pi_{<k, \geq k} \right\vert = \left\vert \Pi_{\geq k, < k} \right\vert = k-1 - \left\vert \Pi_{<k, < k} \right\vert = 1$. Thus $\pi$ satisfies the quadrant condition as desired.

Conversely, suppose that $\pi$ satisfies the quadrant condition, and has no fixed point. Owing to the fact that $\pi$ is irreducible, we have $\vert L_1\vert =0$ (we cannot have $\pi_1 = 1$). Moreover there is always exactly one dot in the rightmost column $n$ and exactly one dot in the bottom-most row $n$, which both belong to the L-subset $L_n$. Again, $\pi$ is irreducible, so these dots must be distinct (we cannot have $\pi_n = n$), leading to the equality $\vert L_n\vert=2$. It therefore remains to show that
\beq\label{eq:L-condition}
\forall j \in \{2, \ldots, n-1\}, \, \vert L_j \vert = 1.
\eeq

Seeking contradiction, suppose that this is not the case. Because $\pi$ is a permutation, there are $n$ dots in total in $\Pi$, and therefore $(n-2)$ dots in $\bigcup\limits_{j = 2}^{n-1} L_j$. Thus if the L-condition~\eqref{eq:L-condition} is not satisfied, there must exist $j \in \{2,\ldots,n-1\}$ such that the $j$-th L-subset $L_j$ is empty. In that case, the dot in the column $j$ will be placed below row $j$, i.e.\ in column $j$ of the lower-left quadrant $\Pi_{<j+1, \geq j+1}$. Now since $L_j$ is empty, each dot located in the first $(j-1)$ columns should be placed either in the upper-left quadrant $\Pi_{<j, <j}$ or the lower-left quadrant $\Pi_{<j+1, \geq j+1}$ but before the column $j$. Because $\pi$ is irreducible, there must be at most $(j-2)$ dots in the upper-left quadrant $\Pi_{<j, <j}$, which means that there must be at least one dot placed in the first $(j-1)$ columns of the lower-left quadrant $\Pi_{<j+1, \geq j+1}$, say in some column $j' < j$. Hence, in the lower-left quadrant $\Pi_{<j+1, \geq j+1}$ there are at least two dots: one in column $j'$ and one in column $j$. This contradicts the fact that $\pi$ satisfies the quadrant condition. Therefore, the L-condition~\eqref{eq:L-condition} is satisfied, and $\pi$ has a unique CNAT by~\cite[Theorem~3.11]{CO}.
\end{proof}

For $n \geq 1$, and an index $j \in \{2,\ldots,n\}$, we now define the \emph{insertion operation} $\Insert[j] : S_n \rightarrow S_{n+1}$ as follows. Given a $n$-permutation $\pi$, we construct $\pi' := \Insert[j](\pi)$ by inserting a fixed point $j$ into the $n$-permutation $\pi$, and increasing the labels of all letters $k \geq j$ in the original permutation $\pi$ by $1$. For example, if $\pi = 521634$ and $j=3$, then $\pi' = 62\textbf{3}1745$ (fixed point in bold). It is straightforward to check that $\Insert[j](\pi)$ is irreducible if, and only if, $\pi$ is irreducible.

\begin{theorem}\label{thm:Insert_B(n,1)toB(n+1,2)}
Let $n \geq 2$. Define a map $\Phi : \{2,\ldots,n\} \times S_{n} \rightarrow S_{n+1}$ by $\Phi(j, \pi) := \Insert[j](\pi)$ for all $j \in \{2,\ldots,n\}$ and $\pi \in S_{n}$. Then $\Phi$ is a bijection from $\{2,\ldots,n\} \times B(n,1)$ to $B(n+1,2)$.
\end{theorem}

\begin{remark}\label{rem:marked_B(n,1)toB(n+1,2)}
The insertion index $j \in \{2, \ldots, n\}$ can be thought of as \emph{marking} the ``space'' between $\pi_{j-1}$ and $\pi_j$ in $\pi$, with the insertion operation corresponding to inserting a fixed point in the marked space. The product set $\{2,\ldots,n\} \times B(n,1)$ can then be thought of as permutations with a single CNAT, marked in such a space. Hence the bijection of Theorem~\ref{thm:Insert_B(n,1)toB(n+1,2)} can be interpreted as a bijection between marked $n$-permutations with a single CNAT, and $(n+1)$-permutations with two CNATs.
\end{remark}

From~\cite[Corollary~3.12]{CO}, we have $b(n,1) = 2^{n-2}$ for any $n \geq 2$. This formula is also implicit from the combination of Theorem~\ref{thm:cnat_tutte} and \cite[Theorem~1]{AH}. Indeed, in that work the authors show that the number of $n$-permutations whose graphs are trees is $2^{n-2}$, and trees are the only simple graphs satisfying $a_G = 1$ (Lemma~\ref{lem:tutte_tree}). Together with Theorem~\ref{thm:Insert_B(n,1)toB(n+1,2)}, this enumeration formula immediately implies the following, which answers in the affirmative the conjecture from~\cite{CO} that $\big( b(n,2) \big)_{n \geq 2}$ is given by Sequence~A001787 in~\cite{OEIS}.

\begin{corollary}\label{cor:b(n,2)}
For any $n\geq2$, we have $b(n,2) = (n-2)\cdot 2^{(n-3)}$.
\end{corollary}

To prove Theorem~\ref{thm:Insert_B(n,1)toB(n+1,2)}, we first state and prove two lemmas on the insertion operation $\Insert[j]$ and the set $B(n,2)$.

\begin{lemma}\label{lem:k_Insert_2k}
Let $n \geq 2$, and $j \in \{2,\ldots,n\}$. For any $k \geq 1$, if $\pi$ is an $n$-permutation with $k$ CNATs, then $\Insert[j](\pi)$ is an $(n+1)$-permutation with at least $2k$ CNATs. Moreover, if $\pi$ has a single CNAT (i.e. $k=1$), then $\Insert[j](\pi)$ has exactly $2$ CNATs.
\end{lemma}

\begin{proof}
Let $T \in \CNAT[\pi]$ be a CNAT with permutation $\pi$. Define $T'$ to be the ``partial'' CNAT obtained by inserting a new dot $d$ in cell $(j,j)$ and shifting all dots in cells $(x,y)$ with $x \geq j$, resp.\ $y \geq j$, one column rightwards, resp.\ one row downwards (see Figure~\ref{fig:partial_cnat}). Since the lower-left and upper-right $j$-quadrants of $\pi$ are non-empty, $T'$ must have at least one leaf dot $d_1$ below and to the left of $d$, and one leaf dot $d_2$ above and to the right of $d$. The path from $d_1$, resp.\ $d_2$, to the root must cross row $j$, resp.\ column $j$, in some column $j_1 < j$, resp.\ row $j_2 < j$. Let $T_1$, resp.\ $T_2$, be $T'$ with an extra dot in the cell $(j_1,j)$ (i.e.\ in row $j$ and column $j_1$), resp.\ in the cell $(j, j_2)$. Then $T_1$ and $T_2$ are two CNATs with permutation $\pi' := \Insert[j](\pi)$, and the maps $T \mapsto T_1$ and $T \mapsto T_2$ are both injective. This shows that if $\pi$ has $k$ CNATs, then $\pi'$ has at least $2k$ CNATs, as desired. 

The above construction is illustrated in Figure~\ref{fig:insert_2k} below. Here we start with a CNAT $T$ and associated permutation $\pi$ (Figure~\ref{fig:cnat_insert_2k}). We then show in Figures~\ref{fig:T1} and \ref{fig:T2} possible choices for the two CNATs $T_1$ and $T_2$ with permutation $\pi' = \Insert[4](\pi)$. Note that here in each case we actually had two choices for the new edge in the CNATs $T_1$ and $T_2$, since there were two edges of $T$ crossing row $4$ to the left of the new leaf dot (in columns $1$ and $2$), and two edges crossing column $4$ above the new leaf dot (in rows $2$ and $3$). 

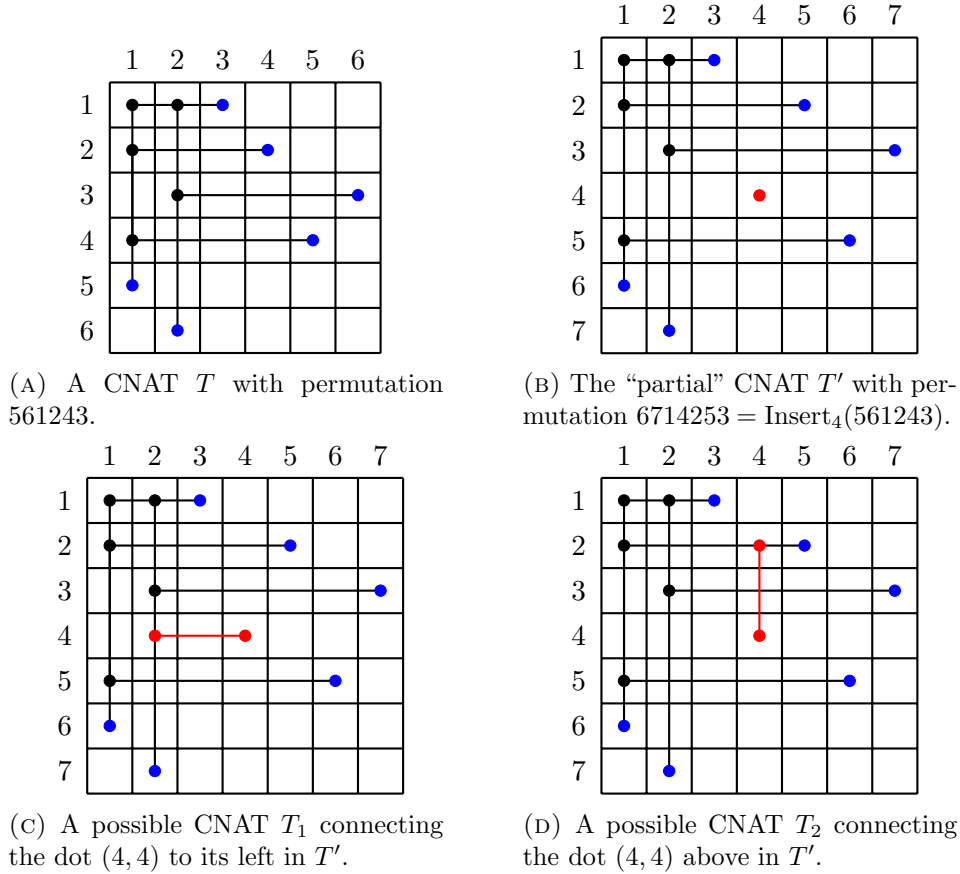
\begin{figure}[ht]
\centering
  \begin{subfigure}[b]{0.35\textwidth}
    \centering
    \begin{tikzpicture}[scale=0.3]
    \draw[step=2cm,thick] (0,0) grid (12,-12.01);
    \draw[thick] (1,-9)--(1,-1)--(5,-1);
    \draw[thick] (3,-1)--(3,-11);
    \draw[thick] (7,-3)--(1,-3)--(1,-7)--(9,-7);
    \draw[thick] (3,-5)--(11,-5);
    \tdot{1}{-9}{blue}
    \tdot{3}{-11}{blue}
    \tdot{5}{-1}{blue}
    \tdot{7}{-3}{blue}
    \tdot{9}{-7}{blue}
    \tdot{11}{-5}{blue}
    \tdot{1}{-1}{black}
    \tdot{1}{-3}{black}
    \tdot{1}{-7}{black}
    \tdot{3}{-1}{black}
    \tdot{3}{-5}{black}
    \foreach \x in {1,...,6}
      \node at (-1+2*\x,1) {$\x$};
    \foreach \x in {1,...,6}
      \node at (-1,1-2*\x) {$\x$};
    \end{tikzpicture}
  \caption{A CNAT $T$ with permutation $561243$.\label{fig:cnat_insert_2k}}
  \end{subfigure}
  \hspace{0.8cm}
  \begin{subfigure}[b]{0.35\textwidth}
    \centering
    \begin{tikzpicture}[scale=0.3]
    \draw[step=2cm,thick] (0,0) grid (14,-14.01);
    \draw[thick] (1,-11)--(1,-1)--(5,-1);
    \draw[thick] (3,-1)--(3,-13);
    \draw[thick] (9,-3)--(1,-3)--(1,-9)--(11,-9);
    \draw[thick] (3,-5)--(13,-5);
    \tdot{1}{-11}{blue}
    \tdot{3}{-13}{blue}
    \tdot{5}{-1}{blue}
    \tdot{9}{-3}{blue}
    \tdot{11}{-9}{blue}
    \tdot{13}{-5}{blue}
    \tdot{7}{-7}{red}
    \tdot{1}{-1}{black}
    \tdot{1}{-3}{black}
    \tdot{1}{-9}{black}
    \tdot{3}{-1}{black}
    \tdot{3}{-5}{black}
    \foreach \x in {1,...,7}
      \node at (-1+2*\x,1) {$\x$};
    \foreach \x in {1,...,7}
      \node at (-1,1-2*\x) {$\x$};
    \end{tikzpicture}
  \caption{The ``partial'' CNAT $T'$ with permutation $6714253 = \Insert[4](561243)$.\label{fig:partial_cnat}}
  \end{subfigure}
  
  \begin{subfigure}[b]{0.35\textwidth}
    \centering
    \begin{tikzpicture}[scale=0.3]
    \draw[step=2cm,thick] (0,0) grid (14,-14.01);
    \draw[thick] (1,-11)--(1,-1)--(5,-1);
    \draw[thick] (3,-1)--(3,-13);
    \draw[thick] (9,-3)--(1,-3)--(1,-9)--(11,-9);
    \draw[thick] (3,-5)--(13,-5);
    \draw[thick, red] (7,-7)--(3,-7);
    \tdot{1}{-11}{blue}
    \tdot{3}{-13}{blue}
    \tdot{5}{-1}{blue}
    \tdot{9}{-3}{blue}
    \tdot{11}{-9}{blue}
    \tdot{13}{-5}{blue}
    \tdot{7}{-7}{red}
    \tdot{1}{-1}{black}
    \tdot{1}{-3}{black}
    \tdot{1}{-9}{black}
    \tdot{3}{-1}{black}
    \tdot{3}{-5}{black}
    \tdot{3}{-7}{red}
    \foreach \x in {1,...,7}
      \node at (-1+2*\x,1) {$\x$};
    \foreach \x in {1,...,7}
      \node at (-1,1-2*\x) {$\x$};
    \end{tikzpicture}
  \caption{A possible CNAT $T_1$ connecting the dot $(4,4)$ to its left in $T'$.\label{fig:T1}}
  \end{subfigure}
  \hspace{0.8cm}
  \begin{subfigure}[b]{0.35\textwidth}
    \centering
    \begin{tikzpicture}[scale=0.3]
    \draw[step=2cm,thick] (0,0) grid (14,-14.01);
    \draw[thick] (1,-11)--(1,-1)--(5,-1);
    \draw[thick] (3,-1)--(3,-13);
    \draw[thick] (9,-3)--(1,-3)--(1,-9)--(11,-9);
    \draw[thick] (3,-5)--(13,-5);
    \draw[thick, red] (7,-7)--(7,-3);
    \tdot{1}{-11}{blue}
    \tdot{3}{-13}{blue}
    \tdot{5}{-1}{blue}
    \tdot{9}{-3}{blue}
    \tdot{11}{-9}{blue}
    \tdot{13}{-5}{blue}
    \tdot{7}{-7}{red}
    \tdot{1}{-1}{black}
    \tdot{1}{-3}{black}
    \tdot{1}{-9}{black}
    \tdot{3}{-1}{black}
    \tdot{3}{-5}{black}
    \tdot{7}{-3}{red}
    \foreach \x in {1,...,7}
      \node at (-1+2*\x,1) {$\x$};
    \foreach \x in {1,...,7}
      \node at (-1,1-2*\x) {$\x$};
    \end{tikzpicture}
  \caption{A possible CNAT $T_2$ connecting the dot $(4,4)$ above in $T'$.\label{fig:T2}}
  \end{subfigure}

\caption{Illustrating how inserting a fixed point into a permutation with $k$ CNATs gives a permutation with at least $2k$ CNATs.\label{fig:insert_2k}}
\end{figure}

Now suppose that $k=1$, i.e.\ that $T$ is the unique CNAT associated with the permutation $\pi$. Let $T'$ be a CNAT associated with $\pi'$. $T'$ has a leaf dot $d$ in cell $(j,j)$ by definition. Moreover, in $T'$, there must be either a dot $d_1$ in a cell $(j_1, j)$ for some $j_1 < j$ (to the left of $d$), or a dot $d_2$ in a cell $(j, j_2)$ for some $j_2 < j$ (above $d$).

First, consider the case where there is a dot $d_1$ to the left of $d$. Any such dot $d_1$ must have a leaf dot below it since the tree $T'$ is complete. By Proposition~\ref{pro:B(n,1)_quadrants}, this leaf dot must be the unique dot in the lower-left quadrant $\Pi'_{<j+1, \geq j+1}$ of $\pi'$ (equivalently, the unique dot in the lower-left quadrant $\Pi_{<j, \geq j}$ of $\pi$). This implies that the dot $d_1$ is unique, i.e.\ that there is only one dot to the left of $d$ in its row. Since $T$ is unique, this means that there is only one CNAT $T'$ with permutation $\pi'$ such that the leaf dot $d$ in cell $(j,j)$ has a dot to its left. Analogously, there is only one CNAT $T'$ with permutation $\pi'$ such that the leaf dot $d$ in cell $(j,j)$ has a dot above it. Since every CNAT $T'$ must be in one of these two cases by definition, this implies that $\cnat[\pi'] = 2$, as desired. 
\end{proof}

\begin{lemma}\label{lem:B(n,2)_quad_cond}
For $n \geq 2$, let $\pi \in B(n,2)$ be an $n$-permutation with exactly $2$ CNATs. Then $\pi$ satisfies the quadrant condition from Definition~\ref{def:quad_cond}.
\end{lemma}

\begin{proof}
Let $\pi$ be an $n$-permutation, and $k \in \{2, \ldots, n\}$ an index. We show that if the lower-left and upper-right $k$-quadrants have at least $2$ points, then we can construct at least $3$ CNATs with associated permutation $\pi$. The construction is similar to the one in the proof of Lemma~\ref{lem:k_Insert_2k}, so we allow ourselves to be a little briefer here.

Let $d_1, d_1'$, resp.\ $d_2, d_2'$, be two dots in the lower-left, resp.\ upper-right, quadrant. We assume that the column of $d_1$ is to the left of that of $d_1'$, and the row of $d_2$ above that of $d_2'$. Let $\pi'$ be the permutation obtained from $\pi$ by removing the rows and columns containing $d_1'$ and $d_2'$ from its plot. Let $T$ be a CNAT with permutation $\pi'$\footnote{See the proof of Lemma~\ref{lem:B(n,2)_quad_cond} for why the assumption that $T$ exists can be justified.}. As in the proof of Lemma~\ref{lem:k_Insert_2k}, we let $T'$ be the ``partial'' CNAT obtained from $T$ by re-inserting the leaf dots $d_1'$ and $d_2'$.

Let $c = (i', j')$ denote the cell in the same column $i'$ as $d_1'$ and row $j'$ as $d_2'$. By construction this cell is in the upper-left $k$-quadrant of $\pi$. Since $d_1$ is in some column $i < i'$ (to the left of $d_1'$), the path from $d_1$ to the root in the partial CNAT $T'$ must cross the row $j'$ in some cell $c_{\mathrm{left}}$ to the left of $c$. Similarly, since the row $j$ of $d_2$ is above the row $j'$ of $d_2'$, the path from $d_2$ to the root must cross the column $i'$ in some cell $c_{\mathrm{up}}$ above $c$. Then putting internal dots in any two of the three cells $c, c_{\mathrm{left}}, c_{\mathrm{up}}$ will yield a CNAT with permutation $\pi$. Since there are three such possibilities, this implies that $\pi$ has at least three CNATs, which completes the proof. These constructions are illustrated in Figure~\ref{fig:B(n,2)_quad_cond} below.

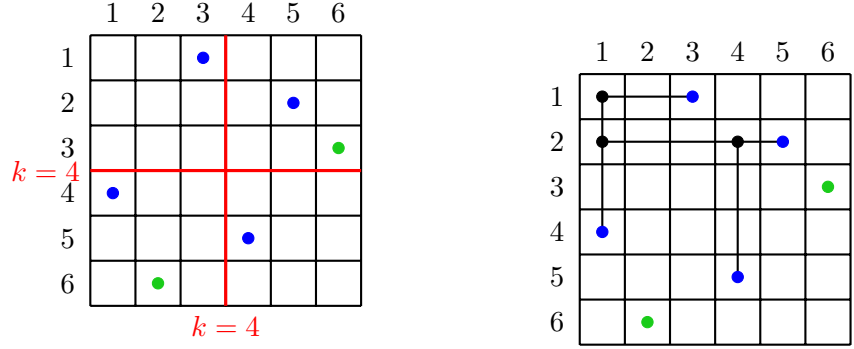
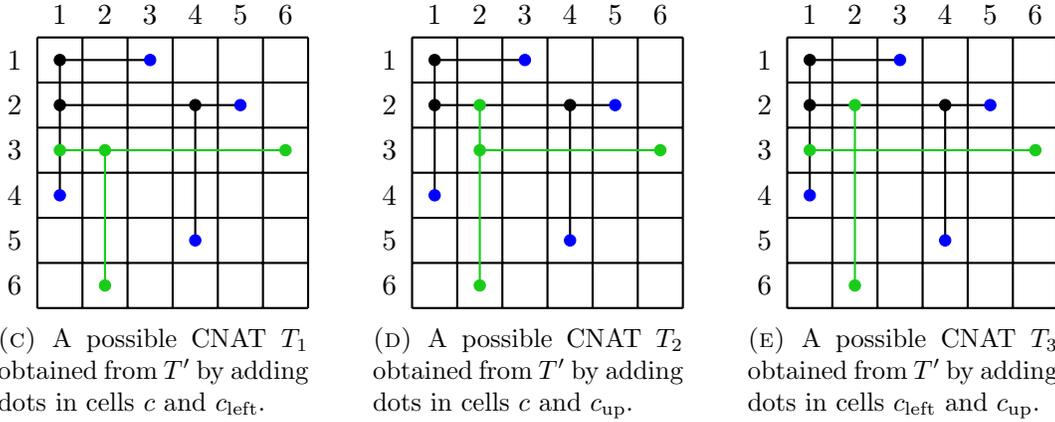
\begin{figure}[ht]
\centering
  \begin{subfigure}[b]{0.35\textwidth}
    \centering
    \begin{tikzpicture}[scale=0.3]
    \draw[step=2cm,thick] (0,0) grid (12,-12.01);
    \tdot{1}{-7}{blue}
    \tdot{3}{-11}{mygreen}
    \tdot{5}{-1}{blue}
    \tdot{7}{-9}{blue}
    \tdot{9}{-3}{blue}
    \tdot{11}{-5}{mygreen}
    \foreach \x in {1,...,6}
    	\node at (-1+2*\x,1) {$\x$};
    \foreach \x in {1,...,6}
    	\node at (-1,1-2*\x) {$\x$};
    \node [left] at (0,-6) {$\textcolor{red}{k=4}$};
    \node [below] at (6,-12) {$\textcolor{red}{k=4}$};
    \draw [red, very thick] (0,-6)--(12,-6);
    \draw [red, very thick] (6,0)--(6,-12);
    \end{tikzpicture}
  \caption{The permutation $\pi = 461523$. The lower-left and upper-right $4$-quadrants have $2$ dots.\label{fig:perm_not_quad_cond}}
  \end{subfigure}
  \hspace{0.8cm}
  \begin{subfigure}[b]{0.35\textwidth}
    \centering
    \begin{tikzpicture}[scale=0.3]
    \draw[step=2cm,thick] (0,0) grid (12,-12.01);
    \draw[thick] (1,-7)--(1,-1)--(5,-1);
    \draw[thick] (7,-9)--(7,-3)--(9,-3);
    \draw[thick] (7,-3)--(1,-3);
    \tdot{1}{-7}{blue}
    \tdot{3}{-11}{mygreen}
    \tdot{5}{-1}{blue}
    \tdot{7}{-9}{blue}
    \tdot{9}{-3}{blue}
    \tdot{11}{-5}{mygreen}
    \tdot{1}{-1}{black}
    \tdot{1}{-3}{black}
    \tdot{7}{-3}{black}
    \foreach \x in {1,...,6}
      \node at (-1+2*\x,1) {$\x$};
    \foreach \x in {1,...,6}
      \node at (-1,1-2*\x) {$\x$};
    \end{tikzpicture}
  \caption{A possible ``partial'' CNAT $T'$ with non-connected dots $d_1'$ in cell $(2,6)$ and $d_2'$ in cell $(6,3)$.\label{fig:partial_not_quad_cond}}
  \end{subfigure}
  
  \hspace{0.2cm}
  
  \begin{subfigure}[b]{0.25\textwidth}
    \centering
    \begin{tikzpicture}[scale=0.3]
    \draw[step=2cm,thick] (0,0) grid (12,-12.01);
    \draw[thick] (1,-7)--(1,-1)--(5,-1);
    \draw[thick] (7,-9)--(7,-3)--(9,-3);
    \draw[thick] (7,-3)--(1,-3);
    \draw[thick, mygreen] (3,-11)--(3,-5)--(11,-5);
    \draw[thick, mygreen] (1,-5)--(3,-5);
   \tdot{1}{-7}{blue}
    \tdot{3}{-11}{mygreen}
    \tdot{5}{-1}{blue}
    \tdot{7}{-9}{blue}
    \tdot{9}{-3}{blue}
    \tdot{11}{-5}{mygreen}
    \tdot{1}{-1}{black}
    \tdot{1}{-3}{black}
    \tdot{7}{-3}{black}
    \tdot{1}{-5}{mygreen}
    \tdot{3}{-5}{mygreen}
    \foreach \x in {1,...,6}
      \node at (-1+2*\x,1) {$\x$};
    \foreach \x in {1,...,6}
      \node at (-1,1-2*\x) {$\x$};
    \end{tikzpicture}
  \caption{A possible CNAT $T_1$ obtained from $T'$ by adding dots in cells $c$ and $c_{\mathrm{left}}$.\label{fig:d_dl}}
  \end{subfigure}
  \hspace{0.6cm}
  \begin{subfigure}[b]{0.25\textwidth}
    \centering
    \begin{tikzpicture}[scale=0.3]
    \draw[step=2cm,thick] (0,0) grid (12,-12.01);
    \draw[thick] (1,-7)--(1,-1)--(5,-1);
    \draw[thick] (7,-9)--(7,-3)--(9,-3);
    \draw[thick] (7,-3)--(1,-3);
    \draw[thick, mygreen] (3,-11)--(3,-5)--(11,-5);
    \draw[thick, mygreen] (3,-3)--(3,-5);
   \tdot{1}{-7}{blue}
    \tdot{3}{-11}{mygreen}
    \tdot{5}{-1}{blue}
    \tdot{7}{-9}{blue}
    \tdot{9}{-3}{blue}
    \tdot{11}{-5}{mygreen}
    \tdot{1}{-1}{black}
    \tdot{1}{-3}{black}
    \tdot{7}{-3}{black}
    \tdot{3}{-3}{mygreen}
    \tdot{3}{-5}{mygreen}
    \foreach \x in {1,...,6}
      \node at (-1+2*\x,1) {$\x$};
    \foreach \x in {1,...,6}
      \node at (-1,1-2*\x) {$\x$};
    \end{tikzpicture}
  \caption{A possible CNAT $T_2$ obtained from $T'$ by adding dots in cells $c$ and $c_{\mathrm{up}}$.\label{fig:d_du}}
  \end{subfigure}
  \hspace{0.6cm}
  \begin{subfigure}[b]{0.25\textwidth}
    \centering
    \begin{tikzpicture}[scale=0.3]
    \draw[step=2cm,thick] (0,0) grid (12,-12.01);
    \draw[thick] (1,-7)--(1,-1)--(5,-1);
    \draw[thick] (7,-9)--(7,-3)--(9,-3);
    \draw[thick] (7,-3)--(1,-3);
    \draw[thick, mygreen] (3,-11)--(3,-3);
    \draw[thick, mygreen] (1,-5)--(11,-5);
   \tdot{1}{-7}{blue}
    \tdot{3}{-11}{mygreen}
    \tdot{5}{-1}{blue}
    \tdot{7}{-9}{blue}
    \tdot{9}{-3}{blue}
    \tdot{11}{-5}{mygreen}
    \tdot{1}{-1}{black}
    \tdot{1}{-3}{black}
    \tdot{7}{-3}{black}
    \tdot{1}{-5}{mygreen}
    \tdot{3}{-3}{mygreen}
    \foreach \x in {1,...,6}
      \node at (-1+2*\x,1) {$\x$};
    \foreach \x in {1,...,6}
      \node at (-1,1-2*\x) {$\x$};
    \end{tikzpicture}
  \caption{A possible CNAT $T_3$ obtained from $T'$ by adding dots in cells $c_{\mathrm{left}}$ and $c_{\mathrm{up}}$.\label{fig:du_dl}}
  \end{subfigure}

\caption{Illustrating how to construct three CNATs with an associated permutation $\pi = 461523$ which does not satisfy the quadrant condition for $k=4$.\label{fig:B(n,2)_quad_cond}}
\end{figure}

However, there is a slight imprecision in the above argument, when we assume that the permutation $\pi'$ obtained by deleting $d_1'$ and $d_2'$ from $\pi$ has a CNAT, which is equivalent to $\pi'$ being irreducible, or equivalently to the graph $G' := G_{\pi'}$ being connected. This may in fact not always be the case (see example below). Denote $(i, \pi_i)$ the cell of the dot $d_1'$, and suppose that deleting $\pi_i$ from $\pi$ (and re-labelling other elements) yields a reducible permutation $\pi^1 \in S_{n-1}$. This means that $\pi^1$ has a breakpoint $j$ for some $1 \leq j < n-1$. Moreover, the permutation $\pi^2 := (\pi_i - j)(\pi_{j+1} - j)(\pi_{j+2} - j) \cdots (\pi_n - j)$ must be irreducible (a breakpoint of $\pi^2$ would also be a breakpoint of $\pi$), and as such has a CNAT, say $T''$. Now in the CNAT construction, when we delete $d_1'$, we also delete all dots in columns $j+1$ to $n$. Then, when re-inserting $d_1'$ into the partial CNAT $T'$, we re-insert those columns by ``gluing'' the CNAT $T''$ into the (full) CNATs. The process is the same when removing $d_2'$, and the rest of the proof is unchanged.

We illustrate this case in Figure~\ref{fig:disconnect_case}, with the permutation $\pi = 4612375$, and $k=4$. The green dots to delete are $d_1'$ in cell $(2,6)$ and $d_2'$ in cell $(5,3)$. However, in this case, deleting the dot $d_1'$ yields $\pi' = (4123)(65)$ which has a breakpoint at $j=4$, so does not have a CNAT. We therefore will also need to delete the dots in columns $6$ and $7$ (in pink). We then construct a partial CNAT $T'$ with associated ``partial permutation'' $412$, resulting from deleting these alongside $d_1'$ and $d_2'$, and a partial CNAT $T''$ with associated ``partial permutation'' $675$ (in pink in Figure~\ref{fig:partial_disconnet}). We then simply ``glue'' $T''$ into the full CNATs that we construct when re-inserting $d_1'$ and $d_2'$ as in the general case.

\begin{figure}[ht]
\centering
  \begin{subfigure}[b]{0.4\textwidth}
    \centering
    \begin{tikzpicture}[scale=0.3]
    \draw[step=2cm,thick] (0,0) grid (14,-14.01);
    \tdot{1}{-7}{blue}
    \tdot{3}{-11}{mygreen}
    \tdot{5}{-1}{blue}
    \tdot{7}{-3}{blue}
    \tdot{9}{-5}{mygreen}
    \tdot{11}{-13}{pink}
    \tdot{13}{-9}{pink}
    \foreach \x in {1,...,7}
    	\node at (-1+2*\x,1) {$\x$};
    \foreach \x in {1,...,7}
    	\node at (-1,1-2*\x) {$\x$};
    \node [left] at (0,-6) {$\textcolor{red}{k=4}$};
    \node [below] at (6,-14) {$\textcolor{red}{k=4}$};
    \draw [red, very thick] (0,-6)--(14,-6);
    \draw [red, very thick] (6,0)--(6,-14);
    \end{tikzpicture}
  \caption{The permutation $\pi = 4612375$. Deleting the dot $d_1'$ in cell $(2, 6)$ would disconnect the pink dots from the remainder of the graph.\label{fig:perm_disconnect}}
  \end{subfigure}
  \hspace{0.8cm}
  \begin{subfigure}[b]{0.4\textwidth}
    \centering
    \begin{tikzpicture}[scale=0.3]
    \draw[step=2cm,thick] (0,0) grid (14,-14.01);
    \draw[thick] (1,-7)--(1,-1)--(5,-1);
    \draw[thick] (7,-3)--(1,-3);
    \draw[thick, pink] (3,-11)--(3,-9)--(13,-9);
    \draw[thick, pink] (11,-9)--(11,-13);
    \tdot{1}{-7}{blue}
    \tdot{3}{-11}{mygreen}
    \tdot{5}{-1}{blue}
    \tdot{7}{-3}{blue}
    \tdot{9}{-5}{mygreen}
    \tdot{11}{-13}{pink}
    \tdot{13}{-9}{pink}
    \tdot{1}{-1}{black}
    \tdot{1}{-3}{black}
    \tdot{3}{-9}{pink}
    \tdot{11}{-9}{pink}
    \foreach \x in {1,...,7}
      \node at (-1+2*\x,1) {$\x$};
    \foreach \x in {1,...,7}
      \node at (-1,1-2*\x) {$\x$};
    \node [below] at (6,-15) {};
    \end{tikzpicture}
  \caption{Possible ``partial'' CNATs $T'$ (black) and $T''$ (pink) with associated ``partial permutations'' $412$ and $675$ respectively.\label{fig:partial_disconnet}}
  \end{subfigure}
  
  \hspace{0.2cm}
  
  \begin{subfigure}[b]{\textwidth}
    \centering
    \begin{tikzpicture}[scale=0.3]
    \draw[step=2cm,thick] (0,0) grid (14,-14.01);
    \draw[thick] (1,-7)--(1,-1)--(5,-1);
    \draw[thick] (7,-3)--(1,-3);
    \draw[thick, pink] (3,-11)--(3,-9)--(13,-9);
    \draw[thick, pink] (11,-9)--(11,-13);
    \draw[thick, mygreen] (3,-9)--(3,-5);
    \draw[thick, mygreen] (9,-5)--(1,-5);
    \tdot{1}{-7}{blue}
    \tdot{3}{-11}{mygreen}
    \tdot{5}{-1}{blue}
    \tdot{7}{-3}{blue}
    \tdot{9}{-5}{mygreen}
    \tdot{11}{-13}{pink}
    \tdot{13}{-9}{pink}
    \tdot{1}{-1}{black}
    \tdot{1}{-3}{black}
    \tdot{3}{-9}{pink}
    \tdot{11}{-9}{pink}
    \tdot{1}{-5}{mygreen}
    \tdot{3}{-5}{mygreen}
    \foreach \x in {1,...,7}
      \node at (-1+2*\x,1) {$\x$};
    \foreach \x in {1,...,7}
      \node at (-1,1-2*\x) {$\x$};
      
    \begin{scope}[shift={(18,0)}]
    \draw[step=2cm,thick] (0,0) grid (14,-14.01);
    \draw[thick] (1,-7)--(1,-1)--(5,-1);
    \draw[thick] (7,-3)--(1,-3);
    \draw[thick, pink] (3,-11)--(3,-9)--(13,-9);
    \draw[thick, pink] (11,-9)--(11,-13);
    \draw[thick, mygreen] (3,-9)--(3,-3);
    \draw[thick, mygreen] (9,-5)--(3,-5);
    \tdot{1}{-7}{blue}
    \tdot{3}{-11}{mygreen}
    \tdot{5}{-1}{blue}
    \tdot{7}{-3}{blue}
    \tdot{9}{-5}{mygreen}
    \tdot{11}{-13}{pink}
    \tdot{13}{-9}{pink}
    \tdot{1}{-1}{black}
    \tdot{1}{-3}{black}
    \tdot{3}{-9}{pink}
    \tdot{11}{-9}{pink}
    \tdot{3}{-3}{mygreen}
    \tdot{3}{-5}{mygreen}
    \foreach \x in {1,...,7}
      \node at (-1+2*\x,1) {$\x$};
    \foreach \x in {1,...,7}
      \node at (-1,1-2*\x) {$\x$};
    \end{scope}
    
    \begin{scope}[shift={(36,0)}]
    \draw[step=2cm,thick] (0,0) grid (14,-14.01);
    \draw[thick] (1,-7)--(1,-1)--(5,-1);
    \draw[thick] (7,-3)--(1,-3);
    \draw[thick, pink] (3,-11)--(3,-9)--(13,-9);
    \draw[thick, pink] (11,-9)--(11,-13);
    \draw[thick, mygreen] (3,-9)--(3,-3);
    \draw[thick, mygreen] (9,-5)--(1,-5);
    \tdot{1}{-7}{blue}
    \tdot{3}{-11}{mygreen}
    \tdot{5}{-1}{blue}
    \tdot{7}{-3}{blue}
    \tdot{9}{-5}{mygreen}
    \tdot{11}{-13}{pink}
    \tdot{13}{-9}{pink}
    \tdot{1}{-1}{black}
    \tdot{1}{-3}{black}
    \tdot{3}{-9}{pink}
    \tdot{11}{-9}{pink}
    \tdot{1}{-5}{mygreen}
    \tdot{3}{-3}{mygreen}
    \foreach \x in {1,...,7}
      \node at (-1+2*\x,1) {$\x$};
    \foreach \x in {1,...,7}
      \node at (-1,1-2*\x) {$\x$};
    \end{scope}
    \end{tikzpicture}
  \caption{Three CNATs that can be constructed by re-inserting $T''$ and $d_2'$ into $T'$. \label{fig:three_cnats_disconnect}}
  \end{subfigure}

\caption{Illustrating how to construct three CNATs with an associated permutation $\pi = 4612375$.\label{fig:disconnect_case}}
\end{figure}
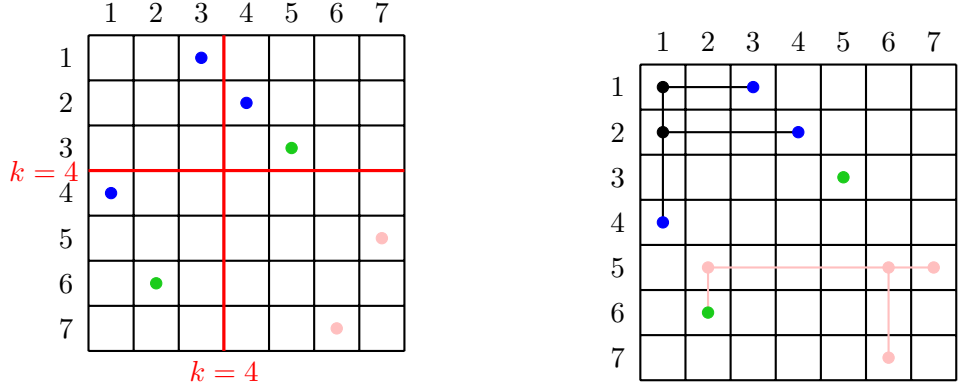
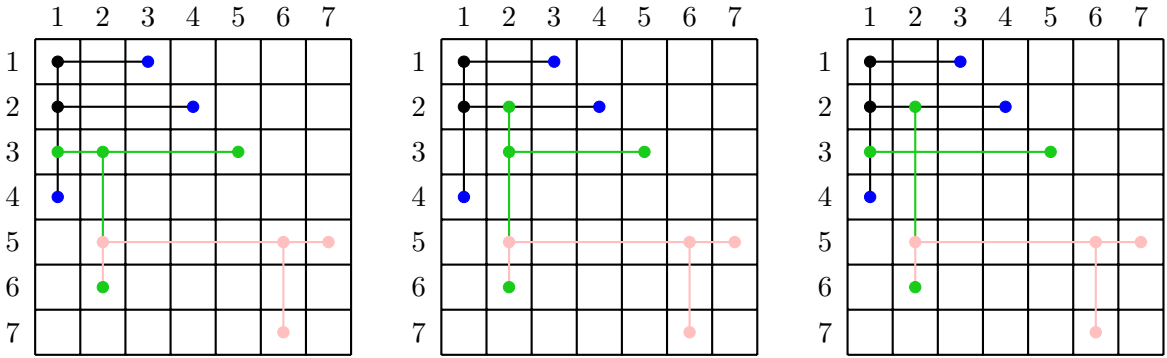

\end{proof}

Combined with Proposition~\ref{pro:B(n,1)_quadrants}, Lemma~\ref{lem:B(n,2)_quad_cond} implies the following.

\begin{lemma}\label{lem:B(n,2)_fixedpoint}
Let $n \geq 2$, and $\pi \in B(n,2)$ an $n$-permutation with exactly $2$ CNATs. Then $\pi$ has a unique fixed point $j \in \{2, \ldots, n-1\}$.
\end{lemma}

\begin{proof}
Suppose $\pi$ is a permutation with $2$ CNATs. Based on Lemma~\ref{lem:B(n,2)_quad_cond}, $\pi$ satisfies the quadrant condition. However, from Proposition~\ref{pro:B(n,1)_quadrants}, if $\pi$ has no fixed point, then it will have a single CNAT. Hence, there must be at least one fixed point in $\pi$. However, if $\pi$ has two fixed points $j$ and $j'$, then we can write $\pi = \Insert[j]\big( \Insert[j'](\pi')\big)$ for some permutation $\pi'$. Moreover, $\pi'$ is irreducible, so has at least one CNAT by~\cite[Theorem~3.3]{CO}. Therefore, by Lemma~\ref{lem:k_Insert_2k}, we then have $\cnat[\pi] \geq 4 \cdot \cnat[\pi'] \geq 4$, which contradicts the fact that $\pi$ has $2$ CNATs.
\end{proof}

Much like Proposition~\ref{pro:B(n,1)_quadrants} for $B(n,1)$, Lemmas~\ref{lem:B(n,2)_quad_cond} and \ref{lem:B(n,2)_fixedpoint} in fact give a characterisation of the set $B(n,2)$, as follows.

\begin{proposition}\label{pro:B(n,2)_characterisation}
Let $\pi$ be a permutation. Then $\pi$ has exactly $2$ CNATs if, and only if, $\pi$ satisfies the quadrant condition and has a unique fixed point $j \in \{2, \ldots, n -1\}$.
\end{proposition}

\begin{proof}
If $\pi$ has exactly $2$ CNATs, then it satisfies the quadrant condition by Lemma~\ref{lem:B(n,2)_quad_cond} and has a unique fixed point $j \in \{2, \ldots, n -1\}$ by Lemma~\ref{lem:B(n,2)_fixedpoint}. Conversely, if $\pi$ satisfies the quadrant condition and has a unique fixed point $j$, then deleting $j$ from $\pi$ yields a permutation $\pi'$ with no fixed point which also satisfies the quadrant condition. Thus $\pi'$ has a unique CNAT by Proposition~\ref{pro:B(n,1)_quadrants}. Since $\pi = \Insert[j](\pi')$ by construction, we deduce that $\pi$ has exactly $2$ CNATs by Lemma~\ref{lem:k_Insert_2k}.
\end{proof}

Theorem~\ref{thm:Insert_B(n,1)toB(n+1,2)} now follows straightforwardly from Propositions~\ref{pro:B(n,1)_quadrants} and \ref{pro:B(n,2)_characterisation}. Indeed, we have shown that permutations with a single CNAT are those which satisfy the quadrant condition and have no fixed point, while permutations with two CNATs are those that satisfy the quadrant condition and have a single fixed point. Since the insertion or deletion of a fixed point does not affect whether a permutation satisfies the quadrant condition or not, it is immediate that the map $\Phi$ defined by $\Phi(j, \pi) := \Insert[j](\pi)$ is indeed a bijection from $\{2,\ldots,n\} \times B(n,1)$ to $B(n+1, 2)$.

\subsection{Linking permutations with two and three CNATs}\label{subsec:B(n,2)_B(n+1,3)}

In this part, we establish a bijection between $n$-permutations with $2$ CNATs and $(n+1)$-permutations with $3$ CNATs (Theorem~\ref{thm:biject_B(n,2)toB(n+1,3)}). For this, we will give characterisations of the sets $B(n,2)$ and $B(n,3)$ in terms of permutation patterns, as introduced at the end of Section~\ref{subsec:perms_prelims}. We begin by stating the following lemma, which will be useful at various points in this section.

\begin{lemma}\label{lem:cnat_pattern}
Suppose that $\pi, \tau$ are permutations such that $\pi$ contains the pattern $\tau$. Then we have $\cnat[\tau] \leq \cnat[\pi]$.
\end{lemma}

\begin{proof}
Since $\pi$ contains the pattern $\tau$, the permutation graph $G_{\pi}$ contains an induced subgraph $H$ which is isomorphic to $G_{\tau}$ by the remarks at the end of Section~\ref{subsec:perms_prelims}. Note that we have $a_{G_{\tau}} = t_H$ in this case. The graph $G_{\pi}$ can then be constructed as follows. Fix some spanning tree $G'$ of $G_{\pi}$. Let $G^0$ be the graph with vertex set $[n] = V\left(G_{\pi}\right)$ and edge set $E\left( G^0 \right) := E(G') \cup E(H)$. By construction, since $G'$ is a tree, we have $\Prune[G^0] = \Prune[H]$, so that $a_{G^0} = t_H$ by Lemma~\ref{lem:tree_pruning_tutte}. Moreover, we have $E\left( G^0 \right) \subseteq E\left(G_{\pi}\right)$, so that the graph $G_{\pi}$ can be obtained from $G^0$ through a (possibly empty) sequence of edge additions. Applying Lemma~\ref{lem:add_edge_tutte} then yields that $a_{G^0} \leq a_{G_{\pi}}$. Combined with the above, we get $a_{G_{\tau}} = t_H \leq a_{G_{\pi}}$, and the result follows from Theorem~\ref{thm:cnat_tutte}.
\end{proof}

\begin{proposition}\label{pro:pattern_B(n,2)}
Let $\pi$ be a permutation. Then $\pi$ has exactly $2$ CNATs if, and only if, $\pi$ contains a unique occurrence of the $321$ pattern and avoids $3412$. Moreover, if $\pi_i, \pi_j, \pi_k$ is the unique occurrence of $321$, with $i < j < k$, then we have $\pi_j = j$, i.e.\ $j$ is the unique fixed point of $\pi$.
\end{proposition}

\begin{proof}
First suppose that $\pi$ contains a unique occurrence of the $321$ pattern, and avoids $3412$. By Proposition~\ref{pro:perm_patterns_cycles}, this implies that the permutation graph $G_{\pi}$ of $\pi$ is a decorated $3$-cycle. From Lemma~\ref{lem:tutte_dec_cycle}, we then obtain $a_G = 3-1 = 2$, which implies by Theorem~\ref{thm:cnat_tutte}, that $\pi$ has exactly $2$ CNATs.

Now suppose that $\pi$ has exactly $2$ CNATs. We know from Lemma~\ref{lem:B(n,2)_fixedpoint} that $\pi$ has a (unique) fixed point $j \in \{2, \ldots, n-1\}$. Consider the lower-left $j$-quadrant $\Pi_{<j, \geq j}$. By Proposition~\ref{pro:B(n,2)_characterisation} it has a unique dot, and this dot cannot be in row $j$, since the dot in row $j$ is in column $j$. Thus this dot must be $(i, \pi_i)$ with $i < j$ and $\pi_i > j = \pi_j$. Similarly, the upper-right $j$-quadrant $\Pi_{ \geq j, < j}$ has a unique dot which is not in column $j$, so is therefore $(k, \pi_k)$ with $k > j$ and $\pi_k < j = \pi_j$. In other words, $\pi_i, \pi_j = j, \pi_k$ forms a $321$ pattern in the permutation $\pi$, or equivalently induces a $3$-cycle in the permutation graph $G_{\pi}$. 

We show that there can be no other occurrences of $321$, or any occurrence of $3412$, in the permutation $\pi$. Otherwise, $G_{\pi}$ would contain another cycle other than that induced on $\pi_i, \pi_j, \pi_k$. In particular this would imply that the $2$-core $\Prune[G_{\pi}]$ is not isomorphic to the $3$-cycle, and therefore we would have $a_{G_{\pi}} > a_{C_3} = 2$ by Lemma~\ref{lem:strict_subgraph}. This completes the proof of the proposition. 
\end{proof}

%

The $321$ pattern structure for $\pi \in B(n,2)$ can be conveniently represented using the terminology of \emph{mesh patterns}. Mesh patterns were introduced by Br\"and\'en and Claesson~\cite{MeshBC} as a generalisation of classical permutation patterns with added restrictions. Formally a \emph{mesh} pattern is a pair $p = (\tau, R)$, with $\tau \in S_k$ and $R \subseteq [0, k] \times [0, k]$ for some positive integer $k$, where $[0,k]$ denotes the integer interval $\{0, \ldots, k\}$. To depict a mesh pattern, we plot the permutation $\tau$ on a $k \times k$ grid (with our standard directions from top to bottom on the rows, and left to right on the columns), and for each $(i,j) \in R$, we shade the unit with top-left corner $(i, j)$. For example, the mesh pattern $p =\left( 3241, \{ (0, 2),(1, 3),(1, 4),(4, 2),(4, 3) \} \right)$ is depicted in Figure~\ref{fig:exa_mesh} below.

 \begin{figure}[ht]
  \centering
  \begin{tikzpicture}[scale=0.4]
    \draw[step=2cm,thick] (0.01,-0.01) grid (9.99,-9.99);
    \fill[pattern=north east lines, pattern color=black!45] (0,-6) rectangle (2,-4);
    \fill[pattern=north east lines, pattern color=black!45] (2,-6) rectangle (4,-10);
    \fill[pattern=north east lines, pattern color=black!45] (2,-4) rectangle (8,-4) rectangle (10,-8);
    \tdot{2}{-6}{blue}
    \tdot{4}{-4}{blue}
    \tdot{6}{-8}{blue}
    \tdot{8}{-2}{blue}
    \node at (2,1) {$1$};
    \node at (4,1) {$2$};
    \node at (6,1) {$3$};
    \node at (8,1) {$4$};
    \node at (-1,-2) {$1$};
    \node at (-1,-4) {$2$};
    \node at (-1,-6) {$3$};
    \node at (-1,-8) {$4$};
  \end{tikzpicture}
  \caption{The depiction of the mesh pattern $p =\left( 3241, \{ (0, 2),(1, 3),(1, 4),(4, 2),(4, 3) \} \right)$.\label{fig:exa_mesh}}
\end{figure}
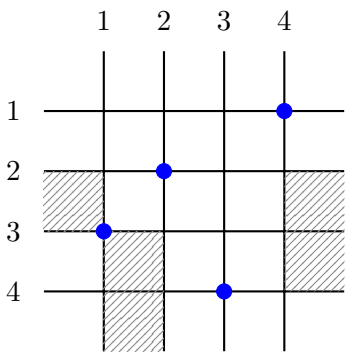

We then say that a permutation $\pi$ contains a mesh pattern $p = (\tau, R)$ if $\pi$ contains the (classical) pattern $\tau$ in such a way that the shaded regions do not contain any elements of $\pi$. For a more formal definition, see~\cite{MeshBC}. For example, the permutation $\pi = 53241$ contains the mesh pattern $p$ from Figure~\ref{fig:exa_mesh}, as the $5$ is not in a forbidden region. However, the permutation $\pi = 35241$ does not contain $p$, as the $5$ is in the forbidden region between columns $1$ and $2$ in Figure~\ref{fig:exa_mesh} (and there is only one occurrence of $\tau = 3241$ in $\pi$). With this terminology, a permutation $\pi \in B(n,2)$ contains a (unique) occurrence of the mesh pattern depicted in Figure~\ref{fig:321_B(n,2)_mesh} below. To check this, note that if any of the shaded regions contained a dot, we would get a second occurrence of the (classical) $321$ pattern. For example, if any region to the left of column $j$ and below row $j$ contained a dot, this would form an occurrence of $321$ with $\pi_j$ and $\pi_k$.

\begin{figure}[ht]
  \centering
  \begin{tikzpicture}[scale=0.4]
    \draw[step=2cm,thick] (0.01,-0.01) grid (7.99,-7.99);
    \fill[pattern=north east lines, pattern color=black!45] (0,-8) rectangle (4,-4);
    \fill[pattern=north east lines, pattern color=black!45] (8,0) rectangle (4,-4);
    \fill[pattern=north east lines, pattern color=black!45] (2,-4) rectangle (4,-2);
    \fill[pattern=north east lines, pattern color=black!45] (4,-6) rectangle (6,-4);
    \tdot{2}{-6}{blue}
    \tdot{4}{-4}{blue}
    \tdot{6}{-2}{blue}
    \node at (2,1) {$i$};
    \node at (4,1) {$j$};
    \node at (6,1) {$k$};
    \node [left] at (-0.5,-2) {$\pi_k$};
    \node [left] at (-0.5,-4) {$\pi_j = j$};
    \node [left] at (-0.5,-6) {$\pi_i$};
  \end{tikzpicture}
  \caption{Illustrating the unique $321$ pattern in $\pi \in B(n,2)$ using mesh patterns.\label{fig:321_B(n,2)_mesh}}
\end{figure}
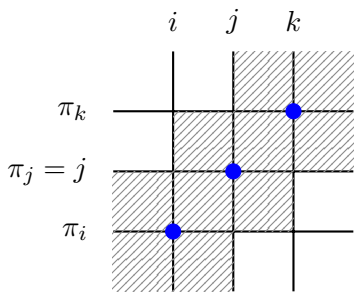

\begin{proposition}\label{pro:pattern_B(n,3)}
Let $\pi$ be a permutation. Then $\pi$ has exactly $3$ CNATs if, and only if, $\pi$ contains a unique occurrence of the $3412$ pattern and avoids $321$. Moreover, if $\pi_i, \pi_j, \pi_k, \pi_{\ell}$ is the occurrence of $3412$, with $i < j < k < \ell$, then we have $\pi_i = \pi_{\ell} + 1$ and $k = j+1$.
\end{proposition}

\begin{proof}

First suppose that $\pi$ contains a unique occurrence of the $3412$ pattern, and avoids $321$. By Proposition~\ref{pro:perm_patterns_cycles}, this implies that the permutation graph $G_{\pi}$ of $\pi$ is a decorated $4$-cycle. From Lemma~\ref{lem:tutte_dec_cycle}, we then obtain $a_{G_{\pi}} = 4-1 = 3$, which implies by Theorem~\ref{thm:cnat_tutte}, that $\pi$ has exactly $3$ CNATs.

We now show the converse. Let $\pi$ be a permutation with exactly $3$ CNATs. We claim that $G := G_{\pi}$ must be a decorated $4$-cycle. Assume for now this claim proved. Then Proposition~\ref{pro:perm_patterns_cycles} implies that $\pi$ must indeed contain a unique occurrence of $3412$ (corresponding to the induced $4$-cycle of $G_{\pi}$), and avoids $321$ (since $G$ does not induce a $3$-cycle). It therefore remains to prove that $G$ is indeed a decorated $4$-cycle in this case.

First, note that if $G$ is a tree, then $\cnat[\pi] = a_G = 1 \neq 3$ (by Lemma~\ref{lem:tutte_tree}), so that $G$ must contain at least one cycle. If $G$ is a decorated $3$-cycle, then by Lemma~\ref{lem:tutte_dec_cycle} we would have $\cnat[\pi] = a_G = 2 \neq 3$. Therefore $G$ either contains at least two cycles, or a cycle of length at least $4$. If $G$ contains a cycle of length $k \geq 5$, then Lemma~\ref{lem:strict_subgraph} implies that $\cnat[\pi] = a_G \geq a_{C_k} = k-1 \geq 4$, a contradiction. Similarly, if $G$ contains a $4$-cycle, to get $a_G = 3$ the $2$-core $\Prune[G]$ must be isomorphic to $C_4$ (again using Lemma~\ref{lem:strict_subgraph}), i.e.\ $G$ is a decorated $4$-cycle.

It therefore remains to show that $G$ cannot contain two $3$-cycles. We again seek contradiction. If $G$ has two $3$-cycles which share an edge, then the ``outer cycle'' is a $4$-cycle (see Figure~\ref{fig:k=4_chord11}). One can check that this graph $G'$ satisfies $a_{G'} = 4$, and again by Lemma~\ref{lem:strict_subgraph} we get the contradictory $\cnat[\pi] = a_G \geq a_{G'} = 4$. It therefore remains to consider the case where $G$ contains at least two edge-disjoint $3$-cycles. Since $G$ is connected, these are necessarily joined in $G$ by a path of $k$ edges for some $k \geq 0$. Let $G'$ be the subgraph of $G$ consisting of the two $3$-cycles together with this joining path (see Figure~\ref{fig:ext_bowtie}). It is straightforward to check that we have $a_{G'} = 4$, and the result follows as above. This completes the proof of our claim that if $\pi$ has exactly $3$ CNATs, the associated permutation graph $G_{\pi}$ is a decorated $4$-cycle, as desired.

\begin{figure}[ht]
\centering
 \begin{tikzpicture}
  \node [draw, circle] (1) at (-0.5,0) {};
  \node [draw, circle] (2) at (1,1) {};
  \node [draw, circle] (3) at (-0.5,2) {};
  \node [draw, circle] (6) at (2.5,1) {};
  \node [draw, circle] (7) at (4,1) {};
  \node [draw, circle] (4) at (5.5,0) {};
  \node [draw, circle] (5) at (5.5,2) {};
  \draw [thick] (7)--(4)--(5)--(7)--(6)--(2)--(3)--(1)--(2);
  \end{tikzpicture}
  \caption{The subgraph $G'$ consisting of two $3$-cycles joined by a path of $k=2$ edges.\label{fig:ext_bowtie}}
\end{figure}
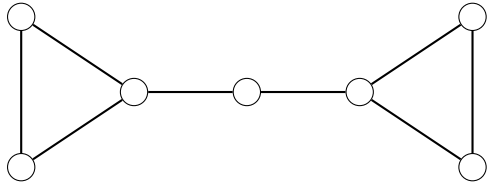

We now prove the second statement of the proposition. Let $\pi_i, \pi_j, \pi_k, \pi_{\ell}$ be the occurrence of the $3412$ pattern in $\pi \in B(n,3)$, with $i < j < k < \ell$. We wish to show that there are no dots between columns $j$ and $k$ (i.e.\ $k = j+1$), and no dots between rows $\pi_{\ell}$ and $\pi_i$ (i.e.\ $\pi_i = \pi_{\ell} + 1$). Let us first show that $k = j+1$, i.e.\ that there are no dots between columns $j$ and $k$. Otherwise, let $m$ be such that $j < m < k$. We show that the dot $(m, \pi_m)$ is either part of an occurrence of the $321$ pattern, or of the $3412$ pattern, which contradicts the above. There are three possible cases to consider.
\begin{enumerate}
\item If $\pi_m < \pi_k$, then $\pi_i, \pi_j, \pi_m, \pi_{\ell}$ is an occurrence of the $3412$ pattern.
\item If $\pi_m > \pi_j$, then $\pi_i, \pi_m, \pi_k, \pi_{\ell}$ is an occurrence of the $3412$ pattern.
\item If $\pi_k < \pi_m < \pi_j$, then $\pi_j, \pi_m, \pi_k$ is an occurrence of the $321$ pattern.
\end{enumerate}

We show similarly that $\pi_i = \pi_{\ell} + 1$. Otherwise, let $m$ be such that $\pi_{\ell} < \pi_m < \pi_i$. As above, there are three cases to consider.
\begin{enumerate}
\item If $m < i$, then $\pi_m, \pi_j, \pi_k, \pi_{\ell}$ is an occurrence of the $3412$ pattern.
\item If $m > \ell$, then $\pi_i, \pi_j, \pi_k, \pi_m$ is an occurrence of the $3412$ pattern.
\item If $i < m < \ell$, then $\pi_i, \pi_m, \pi_{\ell}$ is an occurrence of the $321$ pattern.
\end{enumerate}
As above, these all contradict the fact that $\pi_i, \pi_j, \pi_k, \pi_{\ell}$ is the unique occurrence of $3412$ in $\pi$, and that $\pi$ avoids 321. This concludes the proof.
\end{proof}

As in the $B(n,2)$ case, the $3412$ pattern structure for $\pi \in B(n,3)$ can be conveniently represented using mesh patterns. Indeed, for the occurrence of $3412$ to be unique, a permutation $\pi\in B(n,3)$ must contain a (unique) occurrence of the mesh pattern depicted in Figure~\ref{fig:3412_B(n,3)_mesh} below. Note that the presence of fully shaded regions between columns $j$ and $k$, resp.\ between rows $\pi_{\ell}$ and $\pi_i$, implies that we must have $k = j+1$, resp.\ $\pi_i = \pi_{\ell} + 1$. Similarly to the $B(n,2)$ case, we can check that the presence of a dot in any of the shaded regions would either yield an occurrence of $321$, or another occurrence of $3412$.

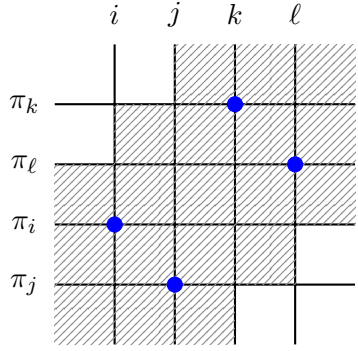
\begin{figure}[ht]
  \centering
  \begin{tikzpicture}[scale=0.4]
    \draw[step=2cm,thick] (0.01,-0.01) grid (9.99,-9.99);
    \fill[pattern=north east lines, pattern color=black!45] (0,-4) rectangle (10,-6);
    \fill[pattern=north east lines, pattern color=black!45] (4,0) rectangle (6,-4);
    \fill[pattern=north east lines, pattern color=black!45] (4,-10) rectangle (6,-6);
    \fill[pattern=north east lines, pattern color=black!45] (0,-10) rectangle (4,-6);
    \fill[pattern=north east lines, pattern color=black!45] (6,0) rectangle (10,-4);
    \fill[pattern=north east lines, pattern color=black!45] (2,-2) rectangle (4,-4);
    \fill[pattern=north east lines, pattern color=black!45] (8,-8) rectangle (6,-6);
    \tdot{2}{-6}{blue}
    \tdot{4}{-8}{blue}
    \tdot{6}{-2}{blue}
    \tdot{8}{-4}{blue}
    \node at (2,1) {$i$};
    \node at (4,1) {$j$};
    \node at (6,1) {$k$};
    \node at (8,1) {$\ell$};
    \node at (-1,-2) {$\pi_k$};
    \node at (-1,-4) {$\pi_{\ell}$};
    \node at (-1,-6) {$\pi_i$};
    \node at (-1,-8) {$\pi_j$};
  \end{tikzpicture}
  \caption{Illustrating the unique $3412$ pattern in $\pi \in B(n,3)$ using mesh patterns.\label{fig:3412_B(n,3)_mesh}}
\end{figure}

\begin{remark}\label{rem:construct_cnats_B(n,3)_pattern}
We can directly construct the three CNATs corresponding to the permutation $\pi$ based on the $3412$ pattern of Figure~\ref{fig:3412_B(n,3)_mesh}. The construction is similar to that of Figure~\ref{fig:B(n,2)_quad_cond}. Namely, we define three cells $c := (j,\pi_{\ell}), c_{\mathrm{left}} := (i,\pi_{\ell}), c_{\mathrm{up}} := (j,\pi_k)$. Deleting the dots $(j, \pi_j)$ and $(\ell, \pi_{\ell})$ from the permutation $\pi$ yields a permutation $\pi'$ which has a single CNAT $T'$ (the permutation graph $G_{\pi'}$ is acyclic, since $\pi'$ avoids $321$ and $3412$)\footnote{As in Proof of Lemma~\ref{lem:B(n,2)_quad_cond} we are slightly abusive here, since we need to take care not to disconnect the permutation graph.}. The three CNATs associated with $T$ are then constructed from $T'$ by adding dots in any two of the three cells $c = (j, \pi_{\ell}),\ c_{\mathrm{left}} = (i, \pi_{\ell}),\ c_{\mathrm{up}} = (j, \pi_k)$, as in Figure~\ref{fig:B(n,2)_quad_cond}.
\end{remark}

As in Section~\ref{subsec:B(n,1)_B(n+1,2)}, it is possible to give a characterisation of $B(n,3)$ in terms of quadrants. We state the result without proof here, since the proof involves a case-by-case study of the numerous possibilities where a permutation does \emph{not} satisfy the stated condition on its quadrants. Indeed, we would have to consider a permutation with two separate values $k, k'$ having two dots in their lower-left and upper-right quadrants, and there are many possibilities for where these dots are located. We consider that such a proof does not add any real value to this paper, so prefer leaving it as an exercise to the dedicated reader.

\begin{proposition}\label{pro:characterisation_B(n,3)_quadrants}
Let $\pi$ be an $n$-permutation for some $n \geq 4$. Then $\pi$ has exactly $3$ CNATs if, and only, if, the following three conditions are all satisfied.
\begin{enumerate}
\item There exists a unique value of $k \in \{2, \ldots, n\}$ such that the lower-left and upper-right $k$-quadrants of $\Pi$ have exactly two dots.
\item All other lower-left and upper-right $k'$-quadrants, for $k' \neq k$, have a single dot.
\item The permutation $\pi$ has no fixed point.
\end{enumerate}
\end{proposition}

With Propositions~\ref{pro:pattern_B(n,2)} and \ref{pro:pattern_B(n,3)}, we are now equipped to define the bijection between $B(n,2)$ and $B(n+1,3)$. Let $\pi \in B(n,2)$ be an $n$-permutation with $2$ CNATs, and suppose that $\pi_i, \pi_j = j, \pi_k$ is the unique occurrence of the pattern $321$ in $\pi$, with $i < j < k$. We first define an $(n+1)$-permutation $\tilde{\pi}$ by $\tilde{\pi} := \Insert[j+1](\pi)$. Since $\pi$ has a unique occurrence of the pattern $321$, given by $\pi_i, \pi_j = j, \pi_k$, it follows by construction that $\tilde{\pi}$ has a unique occurrence of the pattern $4231$, given by $\tilde{\pi}_i = \pi_i + 1, \tilde{\pi}_j = j, \tilde{\pi}_{j+1} = j+1, \tilde{\pi}_{k+1} = \pi_k$. Moreover, by construction, $\tilde{\pi}$  has exactly two occurrences of $321$, given by $\tilde{\pi}_i, \tilde{\pi}_j, \tilde{\pi}_{k+1}$ and $\tilde{\pi}_i, \tilde{\pi}_{j+1}, \tilde{\pi}_{k+1}$. We also claim that $\tilde{\pi}$ avoids $3412$. Indeed, since $\pi$ avoids $3412$, the only way the $3412$ pattern could occur in $\tilde{\pi}$ is by using both $\tilde{\pi}_j = j$ and $\tilde{\pi}_{j+1} = j+1$. This means that $\tilde{\pi}_j \tilde{\pi}_{j+1}$ must be either the $34$ or the $12$ part of the $3412$ pattern. But if it were the $34$ part, then $\pi_i \pi_j x$ could be completed into two occurrences of $321$ in $\pi$ by setting $x$ to be the $1$ or the $2$ in the $3412$ occurrence in $\tilde{\pi}$ (these are the same in $\pi$ and $\tilde{\pi}$). Similarly, if $\tilde{\pi}_j \tilde{\pi}_{j+1}$ were the $12$ part of the $3412$ occurrence in $\tilde{\pi}$, then $y \pi_j \pi_k$ could be completed into two occurrences of $321$ in $\pi$ by setting $y$ to be the $3$ or the $4$ in the $3412$ occurrence in $\tilde{\pi}$. In both cases this contradicts the fact that $\pi$ contains a unique occurrence of $321$, and so the claim that $\tilde{\pi}$ avoids $3412$ is proved. We then define a permutation $\pi'$, obtained from $\tilde{\pi}$ by changing the $4231$ pattern to a $3412$ pattern (on the same elements), and leaving all other elements unchanged. More formally, we have:
\beq\label{eq:def_bij_B(n,2)_B(n,3)}
\begin{cases}
\pi'_i = \tilde{\pi}_{j+1} = j+1 \\
\pi'_j = \tilde{\pi}_i = \pi_i + 1 \\
\pi'_{j+1} = \tilde{\pi}_{k+1} = \pi_k \\
\pi'_{k+1} = \tilde{\pi}_j = j \\
\pi'_{\ell} = \tilde{\pi}_{\ell} \qquad \text{for all other values } \ell
\end{cases}.
\eeq

\begin{theorem}\label{thm:biject_B(n,2)toB(n+1,3)}
Let $n \geq 3$. For a permutation $\pi \in B(n,2)$, define $\pi'$ as in Equation~\eqref{eq:def_bij_B(n,2)_B(n,3)}, where $\pi_i, \pi_j = j, \pi_k$ is the unique occurrence of the pattern $321$ in $\pi$, and $\tilde{\pi} := \Insert[j+1](\pi)$. Then the map $\pi \mapsto \pi'$ is a bijection from $B(n,2)$ to $B(n+1,3)$.
\end{theorem}

The construction $\pi \mapsto \pi'$ is illustrated in Figure~\ref{fig:B(n,2)_B(n+1,3)} below. On the left (Figure~\ref{fig:B(n,2)_321}) we have the occurrence $\pi_i, \pi_j = j, \pi_k$ of $321$ in the permutation $\pi \in B(n,2)$. On the right (Figure~\ref{fig:B(n+1,3)_3412}) we have the occurrence $\pi'_i = \pi_j+1 = j+1, \pi'_j = \pi_i + 1, \pi'_{j+1} = \pi_k, \pi'_{k+1} = \pi_j = j$ of $3412$ in $\pi'$. The other dots of $\pi$ are unchanged in $\pi'$, after the necessary shifting downwards and/or to the right to create a permutation.

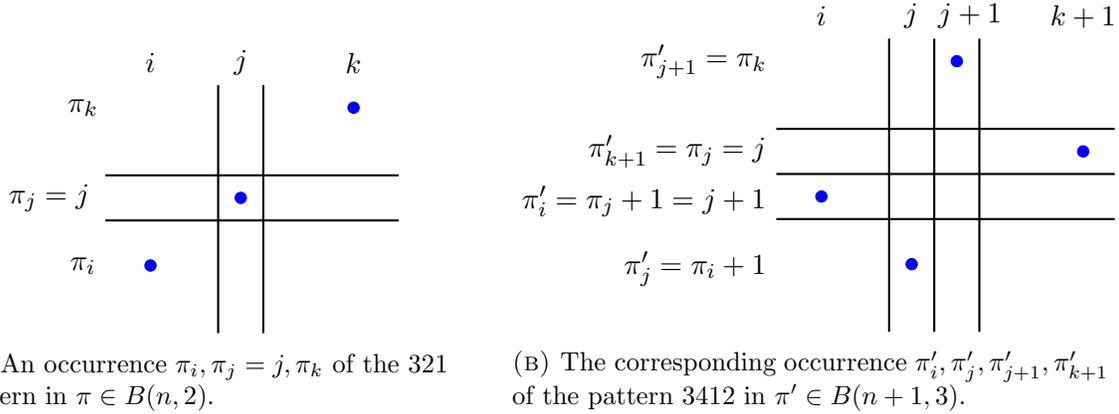
\begin{figure}[ht]
\centering
  \hspace{-0.4cm}
  \begin{subfigure}[b]{0.4\textwidth}
    \centering
    \begin{tikzpicture}[scale=0.3]
    \draw [thick] (0,-4)--(13,-4);
    \draw [thick] (0,-6)--(13,-6);
    \draw [thick] (5,0)--(5,-11);
    \draw [thick] (7,0)--(7,-11);
    \tdot{2}{-8}{blue}
    \tdot{6}{-5}{blue}
    \tdot{11}{-1}{blue}
    \node at (2,1) {$i$};
    \node at (6,1) {$j$};
    \node at (11,1) {$k$};
    \node at (-1,-1) {$\pi_k$};
    \node at (-2.5,-5) {$\pi_j=j$};
    \node at (-1,-8) {$\pi_i$};
    \end{tikzpicture}
  \caption{An occurrence $\pi_i, \pi_j = j, \pi_k$ of the $321$ pattern in $\pi \in B(n,2)$.\label{fig:B(n,2)_321}}
  \end{subfigure}
  \hspace{0.6cm}
  \begin{subfigure}[b]{0.48\textwidth}
    \centering
    \begin{tikzpicture}[scale=0.3]
    \draw [thick] (0,-4)--(15,-4);
    \draw [thick] (0,-6)--(15,-6);
    \draw [thick] (0,-8)--(15,-8);
    \draw [thick] (5,0)--(5,-13);
    \draw [thick] (7,0)--(7,-13);
    \draw [thick] (9,0)--(9,-13);
    \tdot{2}{-7}{blue}
    \tdot{6}{-10}{blue}
    \tdot{8}{-1}{blue}
    \tdot{13.6}{-5}{blue}
    \node at (2,1) {$i$};
    \node [xshift=-0.2] at (6,1) {$j$};
    \node [xshift=5] at (8,1) {$j+1$};
    \node at (13.6,1) {$k+1$};
    \node [left] at (0,-1) {$\pi'_{j+1} = \pi_k$};
    \node [left] at (0,-5) {$\pi'_{k+1} = \pi_{j} = j$};
    \node [left] at (0,-7.2) {$\pi'_i = \pi_j + 1 = j+1$};
    \node [left] at (0,-10.2) {$\pi'_j = \pi_i + 1$};
    \end{tikzpicture}
  \caption{The corresponding occurrence $\pi'_i, \pi'_j, \pi'_{j+1}, \pi'_{k+1}$ of the pattern $3412$ in $\pi' \in B(n+1,3)$.\label{fig:B(n+1,3)_3412}}
  \end{subfigure}

\caption{Illustrating the bijection $\pi \mapsto \pi'$ between $B(n,2)$ and $B(n+1,3)$: the $321$ pattern of $\pi \in B(n,2)$ is mapped to a $3412$ pattern in $\pi'$.\label{fig:B(n,2)_B(n+1,3)}}
\end{figure}

Theorem~\ref{thm:biject_B(n,2)toB(n+1,3)} answers in the affirmative Conjecture~6.4 in~\cite{CO}. Combining with Corollary~\ref{cor:b(n,2)} gives the following.

\begin{corollary}\label{cor:b(n,3)}
For any $n \geq 3$, we have $b(n,3) = (n-3)\cdot 2^{(n-4)}$.
\end{corollary}

\begin{proof}[Proof of Theorem~\ref{thm:biject_B(n,2)toB(n+1,3)}]
Let $\pi \in B(n,2)$. By construction $\pi'$ is an $(n+1)$-permutation. We first show that $\pi' \in B(n+1,3)$, using Proposition~\ref{pro:pattern_B(n,3)}. We first refine the construction of Figure~\ref{fig:B(n,2)_B(n+1,3)}, using the mesh pattern of $\pi$ illustrated in Figure~\ref{fig:321_B(n,2)_mesh}. For convenience, we represent this as a \emph{block decompositions} of the permutations $\pi$ and $\pi'$, as illustrated in Figure~\ref{fig:block_decomp_perm} below. The unlabelled blocks correspond to the shaded regions of the mesh patterns, and are therefore empty.

Now the dots in the upper-left block $A$ cannot form a permutation (since $\pi$ is indecomposable). Therefore if $A$ is non-empty, then $A_1$ or $A_2$ must also be non-empty. However, if there is a dot in $A_1$ and a dot in $A_2$, this would form an occurrence of $3412$ together with $\pi_i$ and $\pi_k$. So at most one of these could be non-empty. Similarly, if the block $B$ is non-empty, then one of $B_1, B_2$ must be non-empty, and moreover at most one of these can be non-empty (even if $B$ is empty). One can check that these block conditions prevent any occurrences of $321$ or $3412$ in $\pi'$, other than the $3412$ already present. By Proposition~\ref{pro:pattern_B(n,3)} we therefore have $\pi' \in B(n+1, 3)$ as desired.

\begin{figure}[ht]
\centering
    \begin{tikzpicture}[scale=0.3]
    \foreach \yy in {-1, -5, -9}
       \draw [thick] (-1.5,\yy)--(13.5,\yy);
    \foreach \xx in {2, 6, 10}
       \draw [thick] (\xx,2)--(\xx,-12);
    \tdot{2}{-9}{blue}
    \tdot{6}{-5}{blue}
    \tdot{10}{-1}{blue}
    \node [above] at (2,2) {$i$};
    \node [above] at (6,2) {$j$};
    \node [above] at (10,2) {$k$};
    \node [left] at (-1.5,-1) {$\pi_k$};
    \node [left] at (-1.5,-5) {$\pi_j=j$};
    \node [left] at (-1.5,-9) {$\pi_i$};
    \node at (0.25, 0.5) {\LARGE{$A$}};
    \node at (0.25, -3) {\LARGE{$A_1$}};
    \node at (4, 0.5) {\LARGE{$A_2$}};
    \node at (11.75, -10.5) {\LARGE{$B$}};
    \node at (8, -10.5) {\LARGE{$B_1$}};
    \node at (11.75, -7) {\LARGE{$B_2$}};
   
   \begin{scope}[shift={(22,0)}]
    \foreach \yy in {-1, -5, -7, -11}
       \draw [thick] (-1.5,\yy)--(15.5,\yy);
    \foreach \xx in {2, 6, 8, 12}
       \draw [thick] (\xx,2)--(\xx,-14);
    \tdot{2}{-7}{blue}
    \tdot{6}{-11}{blue}
    \tdot{8}{-1}{blue}
    \tdot{12}{-5}{blue}
    \node [above] at (2,2) {$i$};
    \node [above, xshift=-0.2] at (6,2) {$j$};
    \node [above, xshift=5] at (8,2) {$j+1$};
    \node [above] at (12,2) {$k$};
    \node [left] at (-1.5,-1) {$\pi_k$};
    \node [left] at (-1.5,-5) {$j$};
    \node [left] at (-1.5,-7) {$j+1$};
    \node [left] at (-1.5,-11) {$\pi_i + 1$};
    \node at (0.25, 0.5) {\LARGE{$A$}};
    \node at (0.25, -3) {\LARGE{$A_1$}};
    \node at (4, 0.5) {\LARGE{$A_2$}};
    \node at (13.75, -12.5) {\LARGE{$B$}};
    \node at (10, -12.5) {\LARGE{$B_1$}};
    \node at (13.75, -9) {\LARGE{$B_2$}};
    \end{scope}
    
    \end{tikzpicture}

\caption{The block decompositions of $\pi \in B(n,2)$ (left) and $\pi' \in B(n+1,3)$ (right). Unlabelled blocks are empty.\label{fig:block_decomp_perm}}
\end{figure}
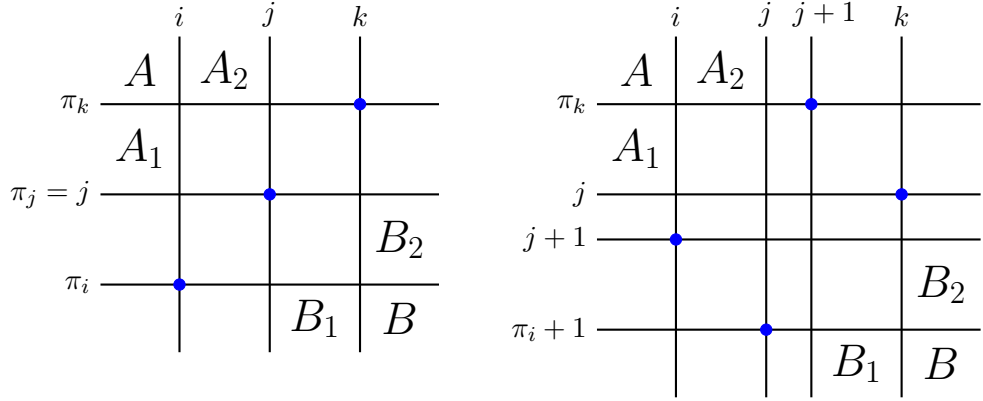

The fact that the map $\pi \mapsto \pi'$ is a bijection is straightforward. Indeed, if $\pi' \in B(n+1, 3)$, then by Proposition~\ref{pro:pattern_B(n,3)} $\pi'$ avoids $321$ and contains a unique occurence $\pi'_i, \pi'_j, \pi'_{j+1}, \pi'_{\ell}$ of $3412$ with $\pi'_i = \pi'_{\ell} + 1$ (see Figure~\ref{fig:3412_B(n,3)_mesh}). We define the corresponding permutation $\pi$ by first replacing this occurrence with an occurrence of $4231$ (i.e.\ $\tilde{\pi}_i = \pi'_j, \tilde{\pi}_j = \pi'_{\ell}, \tilde{\pi}_{j+1} = \pi'_{\ell} + 1, \tilde{\pi}_{\ell} = \pi'_{j+1}$) and then ``contracting'' the two middle dots (with suitable relabelling of the other elements). As above, by considering the block decompositions and using Proposition~\ref{pro:pattern_B(n,2)}, we can show that $\pi$ is in $B(n,2)$, as desired. This concludes the proof.
\end{proof}

\subsection{Permutations with $5$ CNATs}\label{subsec:B(n,5)}

In~\cite[Conjecture~6.5]{CO}, it was conjectured that there are no permutations with exactly $5$ CNATs. We answer this conjecture in the affirmative.

\begin{theorem}\label{thm:b(n,5)=0}
For any $n \geq 1$, we have $b(n,5) = 0$. That is, there are no permutations (of any length) with exactly $5$ CNATs.
\end{theorem}

For this, we use Theorem~\ref{thm:cnat_tutte}. A permutation $\pi$ has exactly $5$ CNATs if, and only if, its permutation graph $G_{\pi}$ has acyclic orientation number equal to $5$. It turns out (see Proposition~\ref{pro:5_tutte}) that essentially only one graph satisfies this enumeration (and that it is not a permutation graph). We begin with a few technical lemmas. We say that a graph $G$ is \emph{two-connected} if for any vertex $v \in G$, the graph $G \setminus \{v\}$, obtained by removing $v$ and any incident edges from $G$, is connected. 

\begin{lemma}\label{lem:t=5_2-conn}
Let $G$ be a graph. If $a_G = 5$, then the $2$-core $\Prune[G]$ is two-connected.
\end{lemma}

\begin{proof}
Let $G' := \Prune[G]$. Suppose that $G'$ has a cut-vertex $v$ (i.e.\ removing $v$ disconnects the graph), and let $G'_1, G'_2, \ldots, G'_k$ be the components of $G'$ joined at $v$ (for some $k \geq 2$). By~\cite[Property~(iv)]{Tutte}, we have $a_{G'} = \prod\limits_{1 \leq i \leq k} a_{G'_i}$. But since $a_{G'} = a_G = 5$ is prime, this implies that there must exist $i$ such that $a_{G_i} = 1$ (in fact all but one of the components must satisfy this). Lemma~\ref{lem:tutte_tree} then implies that $G'_i$ must be a tree, which contradicts the fact that $G'$ is the $2$-core of a graph.
\end{proof}

\begin{lemma}\label{lem:2-conn_cycle}
If $G$ is two-connected with at least $3$ vertices, then every vertex $v$ of $G$ belongs to a cycle.
\end{lemma}

\begin{proof}
A two-connected graph with at least $3$ vertices cannot contain a vertex of degree one, since removing the neighbour of such a vertex would disconnect the graph. Therefore any vertex $v$ in a two-connected graph $G$ must have two distinct neighbours $w_1$ and $w_2$. Now by definition the graph $G \setminus \{v\}$ is connected, so must contain a path from $w_1$ to $w_2$. Connecting this path to the edges $(v, w_1)$ and $(w_2, v)$ in $G$ gives the desired cycle.
\end{proof}

\begin{lemma}\label{lem:4-cycle}
Let $G$ be a two-connected graph with $\vert V(G) \vert \geq 5$ such that the longest cycle in $G$ has length equal to $4$. Let $H$ be the graph consisting of two $4$-cycles glued along two edges (see Figure~\ref{fig:two_4-cyc}). Then $H$ is a subgraph of $G$.
\end{lemma}

\begin{figure}
  \centering
  \begin{tikzpicture}
    \node [draw, circle] (1) at (0,0) {};
    \node [draw, circle] (2) at (1,1) {};
    \node [draw, circle] (3) at (0,2) {};
    \node [draw, circle] (4) at (-1,1) {};
    \node [draw, circle] (5) at (0,1) {};
    \draw [thick] (1)--(2)--(3)--(4)--(1)--(5)--(3);
  \end{tikzpicture}
  \caption{The graph $H$ consisting of two $4$-cycles which share two consecutive edges.\label{fig:two_4-cyc}}
\end{figure}
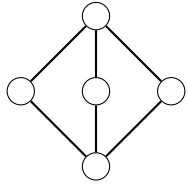

\begin{proof}
Let $C$ be a $4$-cycle in $G$. Consider a vertex $v \notin V(C)$. Since $G$ is two-connected, there must exist two paths from $v$ to distinct vertices of $C$ which do not intersect (except at $v$). Let us denote $P_1$ and $P_2$ these two paths, ending respectively at vertices $v_1$ and $v_2$ of $C$ (with $v_1 \neq v_2$). If $v_1$ and $v_2$ are adjacent vertices in $C$, then $C \setminus \{ (v_1, v_{2}) \} \cup P_1 \cup P_2$ would form a cycle of length at least $5$, a contradiction. Therefore $v_{1}$ and $v_{2}$ must be diagonally opposite vertices in $C$. Taking the union of a half-cycle from $v_{1}$ to $v_{2}$ together with $P_1$ and $P_2$ then gives a cycle of length $2 + \vert E(P_1) \vert + \vert E(P_2) \vert$. Since all cycles in $G$ have length at most $4$, this implies that $P_1$ and $P_2$ each consist of a single edge, and the union of these with $C$ gives the desired subgraph $H$.
\end{proof}

We are now ready to make precise our statement that there is essentially only one graph with acyclic orientation number $5$.

\begin{proposition}\label{pro:5_tutte}
Let $G$ be a simple graph. Then $a_G = 5$ if, and only if, $G$ is a decorated $6$-cycle.
\end{proposition}

\begin{proof}
If $G$ is a decorated $6$-cycle, then we have $a_G = a_{\Prune[G]} = 6-1 = 5$ by Lemma~\ref{lem:tutte_dec_cycle}. Suppose now that $G$ is a simple graph such that $a_G = 5$, and define $G' := \Prune[G]$ to be the $2$-core of $G$. We wish to show that $G'$ is (isomorphic to) the $6$-cycle $C_6$. Let $k$ be the length of the longest cycle in $G'$ (such a cycle must exist since otherwise we would have $a_G = a_{G'} = 1$ by Lemma~\ref{lem:tutte_tree}). That is, $k$ is the maximal value such that $G'$ contains a $k$-cycle $v_0, v_1, \ldots, v_{k-1}, v_k=v_0$ where the first $k-1$ vertices are distinct. With a slight abuse of notation, we denote this $k$-cycle $C_k$. 

We first note that we must have $k \leq 6$. Indeed, otherwise we would have $a_G = a_{G'} \geq a_{C_k} = k-1 \geq 7-1 = 6$, a contradiction (using Lemmas~\ref{lem:tutte_dec_cycle} and \ref{lem:strict_subgraph}). Moreover, if $k=6$, then $G'$ must be the $6$-cycle. Otherwise, $G$ would contain a $6$-cycle to which its $2$-core is not isomorphic, which implies $a_G > a_{C_6} = 5$ by Lemma~\ref{lem:strict_subgraph}.

It therefore remains to show that we cannot have $k \leq 5$. First, note that $G'$ cannot be the $k$-cycle $C_k$ in this case, since otherwise we would have $a_G = a_{G'} = a_{C_k} = k-1 \leq 4$, a contradiction. A \emph{chord} of $C_k$ in $G'$ is a path in $G'$ between two distinct vertices $v_i, v_j$ of $C_k$ which only intersects $C_k$ at the two end-points $v_i$ and $v_j$. We claim that $C_k$ must contain a chord in $G'$. Indeed, since $G'$ is not $C_k$, then there must be an edge $(v, w)$ for some vertex $v \in C_k$ such that $(v,w)$ is not an edge of $C_k$ (i.e.\ if $v = v_i$, then $w \notin \{v_{i-1}, v_{i+1} \}$). If $w \in C_k$, then the edge $(v,w)$ is a chord by definition, so suppose that $w \notin C_k$. Since $G'$ is two-connected by Lemma~\ref{lem:t=5_2-conn}, removing $v$ does not disconnect the graph, so there must be a path from $w$ to one of the neighbours $v'$ of $v$ on the cycle $C_k$ which does not use the edge $(v,v')$. Taking the edge $(v,w)$ together with this path from $w$ to the first point at which it intersects the cycle $C_k$ gives the desired chord. We now proceed on a case-by-case basis on the value of $k$. In each case, the acyclic orientation numbers can be checked e.g.\ through applying Equation~\eqref{eq:Tutte_poly} or through using the SageMath software~\cite{SageMath}.

\begin{description}
\item [Case~$k=3$] Since $G'$ is simple, any chord of $C_3$ must have length at least $2$ (otherwise it would be a multiple edge), let us say that this chord is $v_0, w_1, \ldots, w_k, v_1$, with $k \geq 1$ and $w_i \notin \{v_0, v_1, v_2\}$ for $1 \leq i \leq k$. Then $v_0, w_1, \ldots, w_k, v_1, v_2, v_0$ is a cycle of length at least $4$ in $G'$, a contradiction.
\item [Case~$k=4$] Here we distinguish two possibilities.
  \begin{itemize}
  \item If $\vert V(G') \vert \geq 5$, we apply Lemma~\ref{lem:4-cycle}. One can check that the graph $H$ of Figure~\ref{fig:two_4-cyc} satisfies $t_H = 7$, and by Lemma~\ref{lem:strict_subgraph} we then have $a_{G'} \geq t_H$, which contradicts $a_{G'} = 5$.
  \item If $\vert V(G') \vert = 4$, then $G'$ must be $C_4$ with either one or two chords of length $1$, as in Figure~\ref{fig:cases_4_vertices}. One can check that these have acyclic orientation numbers $4$ and $6$ respectively, neither of which equals $5$.
  \end{itemize}
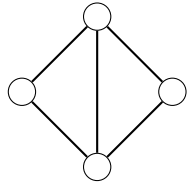
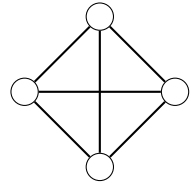
\begin{figure}[ht]
  \centering  
  \begin{subfigure}[b]{0.32\textwidth}
    \centering
    \begin{tikzpicture}
    \node [draw, circle] (1) at (0,0) {};
    \node [draw, circle] (2) at (1,1) {};
    \node [draw, circle] (3) at (0,2) {};
    \node [draw, circle] (4) at (-1,1) {};
    \draw [thick] (1)--(2)--(3)--(4)--(1)--(3);
    \end{tikzpicture}
    \caption{$G'$ is $C_4$ with a chord of length $1$.\label{fig:k=4_chord11}}
  \end{subfigure}
  \hspace{1cm}
  \begin{subfigure}[b]{0.32\textwidth}
    \centering
    \begin{tikzpicture}
    \node [draw, circle] (1) at (0,0) {};
    \node [draw, circle] (2) at (1,1) {};
    \node [draw, circle] (3) at (0,2) {};
    \node [draw, circle] (4) at (-1,1) {};
    \draw [thick] (1)--(2)--(3)--(4)--(1)--(3);
    \draw [thick] (2)--(4);
    \end{tikzpicture}
    \caption{$G'$ is $C_4$ with two chords of length $1$.\label{fig:k=4_chord21}}
  \end{subfigure}
  \caption{The two cases for two-connected graphs on $4$ vertices which are not $C_4$.\label{fig:cases_4_vertices}}  
\end{figure}
\item [Case~$k=5$] Recall that $C_5$ must contain a chord. This chord should have length $1$ or $2$ (otherwise we would get a cycle of length at least $6$), meaning that $G$ contains (as a subgraph) one of the two cases of Figure~\ref{fig:cases_5_vertices}. One can check that the graph $H_1$ satisfies $a_{H_1} = 6$, while the graph $H_2$ yields $a_{H_2} = 10$.\footnote{Contracting a chord edge gives $H_1$, while deleting it gives a decorated $5$-cycle, so $a_{H_2} = a_{H_1} + a_{C_5} = 6 + 4 = 10$.} In particular, Lemma~\ref{lem:strict_subgraph} implies that $a_{G'} \geq \min \{ a_{H_1}, a_{H_2} \} = 6$, contradicting $a_{G'} = 5$. This completes the proof of the proposition.
\begin{figure}[ht]
  \centering  
  \begin{subfigure}[b]{0.32\textwidth}
    \centering
    \begin{tikzpicture}
    \node [draw, circle] (1) at (0,0) {};
    \node [draw, circle] (2) at (1,1) {};
    \node [draw, circle] (3) at (0,2) {};
    \node [draw, circle] (4) at (-1.5,2) {};
    \node [draw, circle] (5) at (-1.5,0) {};
    \draw [thick] (1)--(2)--(3)--(4)--(5)--(1)--(3);
    \end{tikzpicture}
    \caption{$H_1$ is $C_5$ with one chord of length $1$.\label{fig:k=5_chord11}}
  \end{subfigure}
  \hspace{1cm}
  \begin{subfigure}[b]{0.32\textwidth}
    \centering
    \begin{tikzpicture}
    \node [draw, circle] (1) at (0,0) {};
    \node [draw, circle] (6) at (0,1) {};
    \node [draw, circle] (3) at (0,2) {};
    \node [draw, circle] (2) at (1,1) {};
    \node [draw, circle] (4) at (-1.5,2) {};
    \node [draw, circle] (5) at (-1.5,0) {};
    \draw [thick] (1)--(2)--(3)--(4)--(5)--(1)--(6)--(3);
    \end{tikzpicture}
    \caption{$H_2$ is $C_5$ with one chord of length $2$.\label{fig:k=5_chord12}}
  \end{subfigure}
  \caption{The two ``base cases'' for $k=5$.\label{fig:cases_5_vertices}}  
\end{figure}
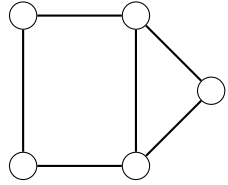
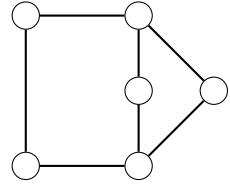
\end{description}
\end{proof}

We are now equipped to prove Theorem~\ref{thm:b(n,5)=0}.

\begin{proof}[Proof of Theorem~\ref{thm:b(n,5)=0}]
Suppose that $G$ is such that $a_G = 5$. By Proposition~\ref{pro:5_tutte}, its $2$-core must be isomorphic to the $6$-cycle. In particular, the original graph $G$ induces a $6$-cycle (the $2$-core of $G$ is always an induced subgraph of $G$). But from Point~(3) of Proposition~\ref{pro:perm_patterns_cycles} this is impossible if $G$ were a permutation graph. As such, we have shown that a graph $G$ satisfying $a_G = 5$ cannot be a permutation graph, which implies that there are no permutations with $5$ CNATs by Theorem~\ref{thm:cnat_tutte}, as desired.
\end{proof}

\begin{problem}\label{pb:proof_b(n,5)}
Our proof of Theorem~\ref{thm:b(n,5)=0} relies heavily on the structure of permutation graphs, and the fact that $\cnat[\pi] =a_{G_{\pi}}$ (Theorem~\ref{thm:cnat_tutte}), which is non-trivial. It would be interesting to find a more direct combinatorial proof of this result.
\end{problem}

\subsection{Permutations with maximal numbers of CNATs}\label{subsec:max_b(n,k)}

Our final result in this section looks at the maximum value of $k$ such that $b(n,k) > 0$. We answer in the affirmative Conjecture~6.3 in~\cite{CO}.

\begin{theorem}\label{thm:max_b(n,k)}
For any $n \geq 1$, we have $\max\{ \cnat[\pi]; \pi \in S_n \} = (n-1)!$, and this maximum is achieved only for the decreasing permutation $\dec[n] = n(n-1)\cdots 1$. In other words, we have $B(n, (n-1)!) = \{ \dec[n] \}$ and $B(n, k) = \emptyset$ for all $k > (n-1)!$.
\end{theorem}

\begin{proof}
That $\cnat[{\dec[n]}] = (n-1)!$ is a consequence of Theorem~\ref{thm:bij_cnat_perm}. We therefore need to show that if $\pi \neq \dec[n]$, we have $\cnat[\pi] < \cnat[{\dec[n]}]$. For this, note that the permutation graph $G_{\dec[n]}$ associated with the decreasing permutation is the complete graph $K_n$ on $n$ vertices, and that $\dec[n]$ is the only permutation whose graph is the complete graph. 
Now for $\pi \neq \dec[n]$, the permutation graph $G_{\pi}$ is a graph on $n$ vertices which is not complete. The complete graph $K_n$ can therefore be obtained from $G_{\pi}$ through a non-empty sequence of edge additions. By Lemma~\ref{lem:add_edge_tutte}, this implies that $a_{G_{\pi}} < a_{K_n} = a_{G_{\dec[n]}}$. The result then follows from Theorem~\ref{thm:cnat_tutte}.
\end{proof}

\begin{remark}\label{rem:b(n,k!)}
Theorem~\ref{thm:max_b(n,k)} implies that $b(n,(n-1)!) = 1$. We have also shown (Corollary~\ref{cor:b(n,2)}) that $b(n,2) = b(n,2!) = (n-2) \cdot 2^{n-3}$, and from~\cite[Corollary~3.12]{CO} we know that $b(n,1) = b(n, 1!) = 2^{n-2}$. Finally, it was also conjectured in~\cite{CO} that $b(n, 6) = b(n, 3!) = \frac{(n-2)(n-3)}{2} \cdot 2^{n-4}$. This suggests that in general, we might have $b(n, k!) = \binom{n-2}{k-1} \cdot 2^{n-1-k} $ for all $n \geq 2$ and $1 \leq k < n$ (Sequence~A038207 in~\cite{OEIS}). However, this formula in fact breaks down almost immediately after these initial values, since our simulations showed that $b(6, 24) = 31$, which is actually a long way off the value of $\binom{4}{3} \cdot 2^1 = 8$ given by the above sequence. In fact, the sequence $\big( b(n, 24) \big)_{n \geq 5}$ does not appear at all in the OEIS, even given just its first three terms $1, 31, 176$. Table~\ref{table:b(n,k!)} below gives the first few rows of the triangular sequence $\big( b(n, k!) \big)_{n \geq 2,\, 1 \leq k < n}$.

\begin{table}[ht]
\centering
\begin{tabular}{|c|c|c|c|c|c|c|}
\hline 
\diagbox{$n$}{$k$}& 1 & 2 & 3 & 4 & 5 & 6 \\
\hline
2 & 1 & & & & & \\ 
3 & 2 & 1 & & & & \\
4 & 4 & 4 & 1 & & & \\
5 & 8 & 12 & 6 & 1 & & \\
6 & 16 & 32 & 24 & 31 & 1 & \\
7 & 32 & 80 & 80 & 176 & 56 & 1 \\
\hline
\end{tabular}
\caption{The values of the sequence $\big( b(n, k!) \big)_{2 \leq n \leq 7,\, 1 \leq k < n}$.\label{table:b(n,k!)}}
\end{table}
\end{remark}


\section{Conclusion and future perspectives}\label{sec:conc}
In this paper, we deepened the combinatorial studies of CNATs initiated in the literature (see e.g.\ ~\cite{ABBS, CO, AvalDet}) by exploiting a connection between CNATs with a given permutation and the acyclic orientation number of the corresponding permutation graph that appeared (in equivalent form) in~\cite{DSSS2}. In Theorem~\ref{thm:bij_cnat_perm}, we provided a new bijection between permutations and so-called upper-diagonal CNATs (CNATs whose associated permutation is the decreasing permutation $n(n-1)\cdots1$), which has the added benefit of preserving certain statistics of these objects. This bijection is defined recursively, based on two operations on labelled CNATs called top-row decomposition and top-row deletion.

We then investigated the enumeration of permutations with a given number of CNATs, answering a number of conjectures from~\cite{CO}. We gave separate characterisations of permutations with exactly $1$, $2$ or $3$ CNATs in terms of so-called \emph{quadrants} of the permutation's plot, and in terms of permutation patterns. This allowed us to establish bijections between marked permutations with a single CNAT and permutations with $2$ CNATs (Theorem~\ref{thm:Insert_B(n,1)toB(n+1,2)}), and between permutations with $2$ CNATs and permutations with $3$ CNATs (Theorem~\ref{thm:biject_B(n,2)toB(n+1,3)}). From this we obtained enumerative formulas for permutations with $2$ or $3$ CNATs (Corollaries~\ref{cor:b(n,2)} and \ref{cor:b(n,3)}). Finally, we showed that there are no permutations of any length with exactly $5$ CNATs (Theorem~\ref{thm:b(n,5)=0}), and that the maximal number of CNATs associated with a given permutation of length $n$ is $(n-1)!$, achieved uniquely for the decreasing permutation $n(n-1)\cdots1$ (Theorem~\ref{thm:max_b(n,k)}).

We end this paper with some possible directions for future research.
\begin{itemize}
\item Find a more ``direct'' combinatorial proof of the fact that there are no permutations with exactly $5$ CNATs, as explained in Problem~\ref{pb:proof_b(n,5)}.
\item Investigate the enumerative sequences $\big( b(n,k) \big)_{n \geq 1}$ for more values of $k$. In this paper, we have looked at the cases $k=1, 2, 3, 5, (n-1)!$. The paper~\cite{CO} suggests that the entries for $k=4, 6, 7$ do indeed appear in the OEIS~\cite{OEIS}. As explained in Remark~\ref{rem:b(n,k!)}, this is not the case for general factorial values of $k$, so it may be that these specific values essentially rely on there being relatively few permutation graphs with the specified acyclic orientation number. For instance, it seems plausible that the only graphs with acyclic orientation number equal to $4$ are the two graphs made up of two triangles in Figures~\ref{fig:ext_bowtie} and \ref{fig:k=4_chord11} with some trees attached. It also would be particularly interesting to know if there are values of $k$ other than $5$ for which $b(n,k)$ is always $0$.
\item One observation from~\cite[Table~1]{CO} which lists the first few values of $\big( b(n,k) \big)$ is that there are very few values for which $b(n,k)$ is odd, other than the $1$'s at the right-hand end of each row implied by Theorem~\ref{thm:max_b(n,k)}. This is also true in our Table~\ref{table:b(n,k!)} above: $b(6, 24) = 31$ is the only other odd value here. Can this observation be quantified in some way?
\item A final direction of future research could be to restrict ourselves to certain subsets of permutations, for example \emph{derangements} (permutations with no fixed point). Fixed points seem to play an important role in the enumeration of permutations according to their number of associated CNATs, since a fixed point implies the existence of a $321$ pattern, which in turn means a $3$-cycle in the permutation graph (and cycles play an important role in graph orientations). For instance, we know that there are no derangements with $2$ CNATs (Proposition~\ref{pro:B(n,2)_characterisation}), while permutations with $1$ and $3$ CNATs are all derangements (Propositions~\ref{pro:B(n,1)_quadrants} and \ref{pro:characterisation_B(n,3)_quadrants}). It is intriguing to consider other values of $k$ and check whether there are derangements or not which have $k$ CNATs. Besides derangements, we could also restrict ourselves to permutations containing or avoiding certain patterns.
\end{itemize}

\section*{Acknowledgments}

The first author would like to thank Einar Steingr\`imsson for helpful initial discussions that led to the results of Section~\ref{sec:upper-diag}. The research leading to these results received funding from the National Natural Science Foundation of China (NSFC) under Grant Agreement No 12101505, by the Research Development Fund of Xi'an Jiaotong-Liverpool University, grant number RDF-22-01-089, and by the Postgraduate Research Scholarship of Xi'an Jiaotong-Liverpool University, grant number PGRS2012026.


\bibliographystyle{plain}
\bibliography{cnat_bibliography}

\end{document}